\documentclass[11p,leqno]{amsart}
\usepackage[letterpaper, total={6in, 8in}, left=1.25in, top=1.5in]{geometry}

\usepackage{amsthm}
\usepackage{amsmath}
\usepackage{amsfonts}
\usepackage{latexsym}
\usepackage{geometry}
\usepackage{enumerate}
\usepackage{csquotes}
\usepackage{xcolor}

\def \b {\beta}

\def\Ric{\text{Ric}}

\def\b{\beta}

\def\R{\mathbb R}

\def\Sph{\mathbb S}

\def\vp{\varphi}

\def\Ric{\operatorname{Ric}}

\numberwithin{equation}{section}

\begin{document}
\title{Ancient solutions to the Ricci flow in higher dimensions}

\author{Xiaolong Li}
\address{Department of Mathematics, University of California, Irvine, Irvine, CA 92697, USA}
\email{xiaolol1@uci.edu}

\author{Yongjia Zhang}
\address{Department of Mathematics, University of Minnesota, Twin Cities, Minneapolis, MN 55414, USA}
\email{zhan7298@umn.edu}


\maketitle
\newtheorem{Theorem}{Theorem}[section]
\newtheorem{Proposition}[Theorem]{Proposition}
\newtheorem{Corollary}[Theorem]{Corollary}
\newtheorem{Lemma}[Theorem]{Lemma}
\newtheorem{Definition}[Theorem]{Definition}
\newtheorem{Conjecture}[Theorem]{Conjecture}

\begin{abstract}
    In this paper, we study $\kappa$-noncollapsed ancient solutions to the Ricci flow with nonnegative curvature operator in higher dimensions $n\geq 4$. We impose one further assumption: one of the asymptotic shrinking gradient Ricci solitons is the standard cylinder $\mathbb{S}^{n-1}\times\mathbb{R}$. 
    First of all, Perelman's structure theorem on three-dimensional ancient $\kappa$-solutions is generalized to all higher dimensions. Secondly, we prove that every noncompact $\kappa$-noncollapsed rotationally symmetric ancient solution to the Ricci flow with bounded positive curvature operator must be the Bryant soliton, thus extending a very recent result of Brendle in three dimensions to all higher dimensions.
\end{abstract}

\tableofcontents

\section{Introduction}

The Ricci flow, a geometric evolution equation introduced by Hamilton \cite{hamilton1982three} in 1982, served as the primary tool in Perelman's solution \cite{perelman2002entropy,Perelman3, perelman2003ricci} of the Poincar\'e and geometrization conjectures, in the resolution of the conjecture of Rauch and Hamilton by B\"ohm and Wilking \cite{bohm2008manifolds}, as well as in the proof of the quarter-pinched differentiable sphere theorem by Brendle and Schoen \cite{BS09}. Further applications of the Ricci flow in understanding the geometry and topology of manifolds in higher dimensions will be the central theme in the study of the Ricci flow.

It is well-known that if one flows an arbitrary metric on a compact manifold, the flow will generally develop singularities. Ancient solutions play a central role in understanding the formation of singularities as they arise naturally in the blowup limits. The study of ancient solutions to the Ricci flow was initiated by Hamilton \cite{hamilton1995formation}, whereas the substantial progress was made by Perelman \cite{perelman2002entropy}. By analyzing the geometry of $\kappa$-noncollapsed ancient solutions, Perelman established a canonical neighborhood theorem for the three-dimensional Ricci flow: every region with high enough curvature in a Ricci flow on a closed three-dimensional manifold should necessarily, after scaling, largely resemble a corresponding piece of a $\kappa$-noncollapsed ancient solution, and hence can only be either a neck, or a cap, or an almost round component. Such a canonical neighborhood theorem made it possible for him to run the Ricci flow after singularities by doing surgeries \cite{perelman2003ricci}. In higher dimensions, ancient solutions were also studies in \cite{brendle2018ricci, brendle2017ricci} and \cite{chen2006ricci} under various curvature conditions, and thereby they established canonical neighborhood theorems and Ricci flows with surgeries under different curvature assumptions.

Perelman \cite{perelman2002entropy} asserted that the only noncompact three-dimensional $\kappa$-noncollapsed ancient Ricci flow with positive sectional curvature is the Bryant soliton. This had been one of the most important conjectures about the three-dimensional Ricci flow and was recently solved by Brendle in \cite{brendle2018ancient}. Such classification of three-dimensional ancient solutions may facilitate the study of four-dimensional Ricci flow, because as in Perelman \cite{perelman2002entropy}, we hope for a proper dimension reduction argument. The corresponding conjecture for the compact case was that a compact and simply-connected ancient solution satisfying all other aforementioned conditions must be a shrinking sphere or Perelman's solution constructed in \cite{perelman2003ricci}; Brendle \cite{brendle2018ancient} proved that such an ancient solution must be rotationally symmetric, with his proof details provided in \cite{brendle2019rotational}. After \cite{brendle2018ancient} was posted, Bamler and Kleiner \cite{bamler2019rotational} also proved the rotational symmetry for such compact and simply-connnected ancient solutions. 
Recently in \cite{ABDS19} and \cite{BDS20}, the authors proved that a $\kappa$-noncollapsed ancient solutions to three-dimensional Ricci flow on $\Sph^3$ is either isometric to a family of shrinking round spheres, or the ancient solution constructed by Perelman.
Therefore, combining \cite{brendle2018ancient} and \cite{BDS20} implies a full classification of ancient $\kappa$-solutions in dimension three: they are the shrinking round sphere, the shrinking round cylinder, Perelman's ancient solution, and the steady Bryant soliton.

One may reasonably ask whether Perelman's conjecture is true in higher dimensions. However, in the case of dimension $n\geq 4$, the geometry of a $\kappa$-noncollapsed ancient Ricci flow becomes complicated, even when assuming nonnegative curvature operator. The reason is that the dimension reduction at infinity gives rise to an $(n-1)$-dimensional $\kappa$-noncollapsed ancient solution, which, in contrast to the case $n=3$, may not be as simple as a shrinking sphere. 
Moreover, many examples of compact ancient solutions, both collapsed and $\kappa$-noncollapsed, in higher dimensions with positive curvature operator were constructed in \cite{BKN12}\cite{Fateev96}\cite{LuWang17}. 
Nevertheless, by adding some reasonable assumption, one might still hope to attack this problem; our Assumption A below provides such a possibility. In particular, as we will see later, a rotationally symmetric $\kappa$-noncollapsed ancient Ricci flow with positive curvature operator always satisfies Assumption A.

Another classification result for ancient Ricci flows was obtained by Daskalopoulos, Hamilton and Sesum \cite{DHS12}. They
proved that any ancient solution on the two-sphere must be either a family of shrinking round spheres or one of the King-Rosenau solutions. 
For closed type I, $\kappa$-noncollapsed ancient solutions to the Ricci flow with positive curvature operator, Ni \cite{ni2009closed} proved that they must be quotients of shrinking round spheres. 
It is worth mentioning that ancient solutions have been studied actively for heat equation and other geometric flows, see \cite{ADS18}\cite{ADS15}\cite{brendle2018uniqueness}\cite{BS19}\cite{Haslhofer15}\cite{LN19}\cite{LZ19}\cite{ni2005ancient} and the references therein.

In this paper, we consider nonflat ancient solutions to the Ricci flow $(M^n,g(t))_{t\in(-\infty,0]}$ with bounded nonnegative curvature operator in dimensions greater or equal to four. Since sometimes we also use backward time, we always let the Latin letter $t$ stand for the forward time and the Greek letter $\tau$ the backward time. We fix $\kappa>0$ and always assume that the ancient solutions we consider are $\kappa$-noncollapsed on all scales. For the sake of convenience we define
\begin{Definition}
An ancient solution to the Ricci flow is called a $\kappa$-solution if it is nonflat, with bounded nonnegative curvature operator and $\kappa$-noncollapsed on all scales. 
\end{Definition}

Furthermore, we impose the following assumption:
\\

\noindent\textbf{Assumption A}: \emph{An} asymptotic shrinking gradient Ricci soliton (in the sense of Perelman \cite[Section 11.2]{perelman2002entropy}) is the standard cylinder $\mathbb{S}^{n-1}\times\mathbb{R}$.
\\

\emph{Remark.} In Perelman \cite{perelman2002entropy}, the asymptotic shrinking gradient Ricci soliton is obtained by using the monotonicity formula of the reduced volume based at some fixed space-time point. We also call such space-time point the \emph{base point of the asymptotic shrinker} (see Theorem \ref{Thm:asymptotic_shrinker} below). For the sake of convenience, when we say that a $\kappa$-solution $(M,g(t))_{t\in(-\infty,0]}$ (or sometimes $(M,g(\tau))_{\tau\in[0,\infty)}$ and $\tau$ stands for the backward time) satisfies Assumption A, we always assume that the base point $(x_0,t_0)$ (or $(x_0,\tau_0)$) of the asymptotic shrinker $\mathbb{S}^{n-1}\times\mathbb{R}$ satisfies $t_0>0$ (or $\tau_0<0$), since we can always shift the base time to make our case so.
\\

Our idea of introducing Assumption A is inspired by Lemma 3.1 in \cite{yokota2009perelman}: given Assumption A, we have that the asymptotic shrinker based at any point in $M\times(-\infty, 0]$ must have Gaussian density no less than that of $\mathbb{S}^{n-1}\times\mathbb{R}$, and hence can only be $\mathbb{S}^{n-1}\times\mathbb{R}$ or $\mathbb{S}^n$ (see Section 2 for more details). This, as we shall see, largely restricts the geometry of the $\kappa$-solutions.

It turns out that under Assumption A, many nice properties of $\kappa$-solutions in dimension three can be extended to higher dimensions. We summarize the most important ones below.

\begin{enumerate}[(1)]
\item Asymptotically cylindrical at space infinity.
\item $\kappa$-compactness theorem (instead of the precompactness theorem of Perelman).
\item Bounded geometry for non-neck-like region in noncompact solutions with positive curvature operator.
\item Neck-stability of Kleiner-Lott \cite{kleiner2014singular} (Theorem \ref{Thm:neck_stability} below).
\end{enumerate}

\begin{Theorem}[Asymptotically cylindrical]\label{Thm:asymptotic_cylindrical}
Let $(M^n,g(t))_{t\in(-\infty,0]}$ be a $\kappa$-solution satisfying Assumption A. For any $p_k\rightarrow\infty$, the sequence $\{(M,g_k(t),p_k)_{t\in(-\infty,0]}\}_{k=1}^\infty$ converges in the Cheeger-Gromov-Hamilton sense, after possibly passing to a subsequence, to the shrinking cylinder $\mathbb{S}^{n-1}\times\mathbb{R}$, where $\displaystyle g_k(t)=Q_kg(tQ_k^{-1})$ and $Q_k=R(p_k,0)$.
\end{Theorem}

\begin{Theorem}[$\kappa$-compactness] \label{Thm:kappa_compactness}
Let $\{(M_k^n,g_k(t),p_k)_{t\in(-\infty,0]}\}_{k=1}^\infty$ be a sequence of $\kappa$-solutions satisfying Assumption A. Let $Q_k=R_k(p_k,0)$ and $\bar{g}_k(t)=Q_kg_k(tQ_k^{-1})$. Then $\{(M_k^n,\bar{g}_k(t),p_k)_{t\in(-\infty,0]}\}_{k=1}^\infty$ converges in the Cheeger-Gromov-Hamilton sense, after possibly passing to a subsequence, to a $\kappa$-solution $(M_\infty,g_\infty(t),p_\infty)_{t\in(-\infty,0]}$. Furthermore, $(M_\infty,g_\infty(t),p_\infty)_{t\in(-\infty,0]}$ satisfies the following condition: either the asymptotic shrinker based at any point in $M_\infty\times(-\infty,0]$ is $\mathbb{S}^{n-1}\times\mathbb{R}$, or $(M_\infty,g_\infty(t))_{t\in(-\infty,0]}$ is the standard shrinking sphere $\mathbb{S}^n$. In particular, if the limit splits, it must be $\mathbb{S}^{n-1}\times\mathbb{R}$.
\end{Theorem}

\emph{Remark.} By Perelman \cite{perelman2002entropy}, without Assumption A, the normalized sequence $\bar{g}_k(t)$ will still converge to an ancient solution $g_\infty(t)$, but $g_\infty(t)$ may not have bounded curvature, hence may not be a $\kappa$-solution.

\begin{Theorem}\label{Thm:Size_of_cap}
For any $\varepsilon>0$, there exists $C(\varepsilon)<\infty$, depending also on $n$ and $\kappa$, such that the following holds. Let $(M^n,g(t))_{t\in(-\infty,0]}$ be a $\kappa$-solution satisfying Assumption A. Furthermore, assume that $(M^n,g(t))$ is noncompact with positive curvature operator. Let $M_\varepsilon$ be the points that are not centers of $\varepsilon$-necks at $t=0$. Then we have $diam(M_{\varepsilon})\leq CQ^{-\frac{1}{2}}$, $C^{-1}Q\leq R(x,0)\leq CQ$ for all $x\in M_{\varepsilon}$, where $Q=R(x_0,0)$ for any point $x_0\in M_\varepsilon$.
\end{Theorem}

The results that we obtained above enable us to generalize part one of \cite{brendle2018ancient} to higher dimensions.

\begin{Theorem}\label{Thm:symmetric}
Let $(M^n,g(t))_{t\in(-\infty,0]}$ be a noncompact $\kappa$-solution with positive curvature operator. Furthermore, assume that $(M, g(t))$ is rotationally symmetric. Then $(M, g(t))$ is isometric to the Bryant soliton up to scaling.
\end{Theorem}

Theorem \ref{Thm:symmetric} together with the strong maximum principle implies
\begin{Theorem}\label{Thm:RS}
Let $(M^n,g(t))_{t\in(-\infty,0]}$ be a noncompact nonflat $\kappa$-solution with nonnegative curvature operator. Furthermore, assume that $(M, g(t))$ is rotationally symmetric. Then $(M, g(t))$ is isometric to either the Bryant soliton up to scaling, or to a family of shrinking cylinders (or a quotient thereof).
\end{Theorem}

\emph{Remark.} Theorem \ref{Thm:symmetric} and \ref{Thm:RS} remain valid if one replaces positive (nonnegative) curvature operator with the weaker condition that $M\times \R$ has positive (nonnegative) isotropic curvature. 
In fact, any ancient solution with weakly PIC1 (namely $M \times \R$ has nonnegative isotropic curvature) has nonnegative complex sectional curvature (also known as weakly PIC2). This was proved in \cite{BCW19} for the bounded curvature case and in \cite{LN19} for the possibly unbounded curvature case. 
Our arguments in this paper then go though if one weakens positive (nonnegative) curvature operator to positive (nonnegative) complex sectional curvature.

The main technical part in the proof of Theorem \ref{Thm:symmetric} is due to Brendle and Choi \cite{brendle2018uniqueness} and Brendle \cite{brendle2018ancient}. Our contribution here is to extend the properties that are satisfied by three-dimensional $\kappa$-solutions (asymptotically cylindrical, neck-stability, etc.) to higher dimensions. Furthermore, we construct the barriers for higher dimensions following \cite{brendle2018ancient}. Note that in Theorem \ref{Thm:symmetric} we do not assume that the ancient solution satisfies Assumption A. Indeed, we will show that in the rotationally symmetric case, Assumption A is automatically satisfied. 

Finally, we would like to mention that part two of \cite{brendle2018ancient} has been extended to higher dimensions by Brendle and Naff \cite{BK20}.





This paper is organized as follows. In Section 2 we show that the sphere $\mathbb{S}^n$ has the highest Gaussian density among all nonflat Ricci shrinkers with nonnegative curvature operator, and that the cylinder $\mathbb{S}^{n-1}\times\mathbb{R}$ has the highest Gaussian density among all noncompact nonflat Ricci shrinkers with nonnegative curvature operator. In Section 3 we prove Theorem \ref{Thm:asymptotic_cylindrical}, Theorem \ref{Thm:kappa_compactness}, and Theorem \ref{Thm:Size_of_cap}. 
In Section 4, we extend the barrier construction in \cite{brendle2018ancient} to higher dimensions and prove Theorem
\ref{Thm:symmetric}.

\section{The Gaussian density of Ricci shrinkers}

A shrinking gradient Ricci soliton (or Ricci shrinker for short) is a triple $(M^n,g,f)$, where $(M^n,g)$ is a smooth complete Riemannian manifold and $f$ is a smooth function on $M$ called the \emph{potential function}, satisfying the equation
\begin{eqnarray}\label{eq:shrinker}
\Ric+\nabla^2f=\frac{1}{2}g.
\end{eqnarray}
By adding to $f$ a constant, we can always normalize the potential function in the way that
\begin{eqnarray}\label{eq:normalize}
|\nabla f|^2+R=f.
\end{eqnarray}
We define the Gaussian density as the following.
\begin{Definition}
Let $(M^n,g,f)$ be a Ricci shrinker normalized as in (\ref{eq:shrinker}) and (\ref{eq:normalize}), then the Gaussian density is the quantity
\begin{eqnarray*}
\tilde{\mathcal{V}}(M)=(4\pi)^{-\frac{n}{2}}\int_Me^{-f}d\mu_g.
\end{eqnarray*}
\end{Definition}
As shown in \cite{cao2010complete}, the potential of a noncompact Ricci shrinker has quadratic growth and the volume has at most Euclidean growth (see also \cite{munteanu2009volume}), thus the Gaussian density of a Ricci shrinker is always well-defined. Our definition is the same as Cao-Hamilton-Ilmanen \cite{cao2004gaussian}, and the Gaussian density of an asymptotic shrinker is the same as the \emph{asymptotic reduced volume}; see (\ref{eq:asymptotic_reduced_volume}) and Theorem \ref{Thm:asymptotic_shrinker}. The main result of this section is the following theorem.

\begin{Theorem}\label{Thm:density}
Among all the $n$-dimensional nonflat Ricci shrinkers with nonnegative curvature operator, $\mathbb{S}^n$ has the highest Gaussian density. Among all the $n$-dimensional noncompact nonflat Ricci shrinkers with nonnegative curvature operator, $\mathbb{S}^{n-1}\times\mathbb{R}$ has the highest Gaussian density.
\end{Theorem}

\emph{Remark.} By \cite{LN19}, Theorem \ref{Thm:density} holds under the weaker assumption that $M\times \R$ has nonnegative isotropic curvature.

The following simple observation indicates that we need only to consider simply connected Ricci shrinkers. Notice that by Wylie \cite{wylie2008complete}, every Ricci shrinker has finite fundamental group.

\begin{Lemma}
Let $(M,g,f)$ be a Ricci shrinker with fundamental group $\Gamma$ and let $(\tilde{M},\tilde{g}, \tilde{f})$ be its universal cover. Then
\begin{eqnarray*}
\tilde{\mathcal{V}}(M)=\frac{1}{|\Gamma|}\tilde{\mathcal{V}}(\tilde{M}).
\end{eqnarray*}
\end{Lemma}

We first deal with the compact case.

\begin{Lemma}\label{Lm:density_compact}
Let $(M^n,g,f)$ be a compact simply-connected Ricci shrinker with nonnegative curvature operator. Then $\tilde{\mathcal{V}}(M)\leq\tilde{\mathcal{V}}(\mathbb{S}^n)$. Furthermore, the equality holds if and only if $M$ is isometric to $\mathbb{S}^n$.
\end{Lemma}
\begin{proof}
By Corollary 4 in \cite{munteanu2017positively}, we have that $(M^n,g)$ is a symmetric space. We claim that $g$ is an Einstein metric. It suffices to show that $f$ is a constant. Suppose this is not true, then let $p_1$ and $p_2$ be the points where $f$ attains its maximum and minimum, respectively. In particular $f(p_1)>f(p_2)$. By (\ref{eq:normalize}), we have
\begin{eqnarray*}
f(p_1)=R(p_1), & f(p_2)=R(p_2).
\end{eqnarray*}
Since $(M,g)$ is a symmetric space, we have that $R$ is a constant, hence $f(p_1)=f(p_2)$; this is a contradiction.

Next, we have that by (\ref{eq:shrinker}) and (\ref{eq:normalize}), both on $M$ and $\mathbb{S}^n$, it holds that
\begin{eqnarray*}
\Ric=\frac{1}{2}g, & f=R=\frac{n}{2}.
\end{eqnarray*}
It then follows from the Bishop-Gromov volume comparison theorem that $\tilde{\mathcal{V}}(M)\leq\tilde{\mathcal{V}}(\mathbb{S}^n)$ and the equality holds if and only if $(M^n,g)$ is isometric to $\mathbb{S}^n$.
\end{proof}

The following lemma is used to deal with the noncompact case.
\begin{Lemma}\label{Lm:density_splitting}
Assume a Ricci shrinker $(M^n,g,f)$ splits as the product of two Ricci shrinkers $(M_1^{n_1},g_1,f_1)$ and $(M_2^{n_2},g_2,f_2)$, where $g=g_1+g_2$, $n=n_1+n_2$, and $f=f_1+f_2$. Then
\begin{eqnarray*}
\tilde{\mathcal{V}}(M)=\tilde{\mathcal{V}}(M_1)\tilde{\mathcal{V}}(M_2).
\end{eqnarray*}
\end{Lemma}
\begin{proof}
This is an easy consequence of Fubini's theorem.
\end{proof}

Now we need to use the actual values of $\tilde{\mathcal{V}}(\mathbb{S}^n)$, one may refer to \cite{cao2004gaussian} or \cite{yokota2009perelman} for the following result.
\begin{eqnarray}\label{eq:density_sn}
\tilde{\mathcal{V}}(\mathbb{S}^n)&=&\int_{\mathbb{S}^n}(4\pi)^{-\frac{n}{2}}e^{-\frac{n}{2}}dg_{\mathbb{S}^n}
\\\nonumber
&=&\frac{\sqrt{2\pi}m^{m+\frac{1}{2}}e^{-m}}{\Gamma(m+1)}\sqrt{\frac{2}{e}},
\end{eqnarray}
where $m=\frac{n-1}{2}$ and $\Gamma$ stands for the the gamma function. By the Stirling's formula
\begin{eqnarray*}
\Gamma(m+1)=\sqrt{2\pi}m^{m+\frac{1}{2}}e^{-m}e^{\theta(m)},
\end{eqnarray*}
where $\theta(m)\searrow 0$ as $m\nearrow\infty$, we have that $\tilde{\mathcal{V}}(\mathbb{S}^n)$ is a strictly increasing sequence in $n$, and its limit is $\displaystyle\sqrt{\frac{2}{e}}$.

With these preparations we are ready to prove Theorem \ref{Thm:density}.

\begin{proof}[Proof of Theorem \ref{Thm:density}]
Let $(M,g,f)$ be an arbitrary $n$-dimensional simply connected Ricci shrinker with nonnegative curvature operator. If $M$ is compact, then the conclusion follows from Lemma \ref{Lm:density_compact}. If $M$ is noncompact, by Corollary 4 in \cite{munteanu2017positively} we have that $M$ is the product of a Gaussian shrinker $\mathbb{R}^{n-k}$ and a compact symmetric space $N^k$, where $2\leq k\leq n-1$. Then by Lemma \ref{Lm:density_splitting} and Lemma \ref{Lm:density_compact} we have that
\begin{eqnarray*}
\tilde{\mathcal{V}}(M)&=&\tilde{\mathcal{V}}(R^{n-k})\cdot\tilde{\mathcal{V}}(N^k)
\\
&=&1\cdot\tilde{\mathcal{V}}(N^k)\leq\tilde{\mathcal{V}}(\mathbb{S}^k)
\\
&\leq&\tilde{\mathcal{V}}(\mathbb{S}^{n-1})\Big(=\tilde{\mathcal{V}}(\mathbb{S}^{n-1}\times\mathbb{R})\Big)
\\
&<&\tilde{\mathcal{V}}(\mathbb{S}^n).
\end{eqnarray*}
This finishes the proof of the theorem.
\end{proof}

\begin{Corollary}[A gap theorem]
There exist positive constants $\varepsilon_n$, depending only on the dimension $n$, such that the following holds. (1) Assume that $(M^n,g,f)$ is a nonflat Ricci shrinker with nonnegative curvature operator and is not isometric to $\mathbb{S}^n$, then $\tilde{\mathcal{V}}(M)\leq\tilde{\mathcal{V}}(\mathbb{S}^n)-\varepsilon_n$. (2) Assume that $(M^n,g,f)$ is a nonflat noncompact Ricci shrinker with nonnegative curvature operator and is not isometric to $\mathbb{S}^{n-1}\times\mathbb{R}$, then $\tilde{\mathcal{V}}(M)\leq\tilde{\mathcal{V}}(\mathbb{S}^{n-1}\times\mathbb{R})-\varepsilon_{n-1}$.
\end{Corollary}
\begin{proof}
We first show that (2) follows from (1). Since $M$ is noncompact, by Corollary 4 of \cite{munteanu2017positively} we can assume that $M=\mathbb{R}\times N^{n-1}$, where $N^{n-1}$ is not $\mathbb{S}^{n-1}$. Then we have
\begin{eqnarray*}
\tilde{\mathcal{V}}(M)&=&\tilde{\mathcal{V}}(N^{n-1})\leq\tilde{\mathcal{V}}(\mathbb{S}^{n-1})-\varepsilon_{n-1}
\\
&=&\tilde{\mathcal{V}}(\mathbb{S}^{n-1}\times\mathbb{R})-\varepsilon_{n-1}.
\end{eqnarray*}

Next we prove (1) by contradiction. Suppose (1) is not true, then we can find a sequence of Ricci shrinkers with nonnegative curvature operator $\{(M_k,g_k,f_k)\}_{k=1}^\infty$, all different from $\mathbb{S}^n$, such that
\begin{eqnarray*}
\tilde{\mathcal{V}}(M_k)\nearrow\tilde{\mathcal{V}}(\mathbb{S}^{n})
\end{eqnarray*}
as $k\rightarrow\infty$. For all $k$ large enough we have
\begin{eqnarray*}
\tilde{\mathcal{V}}(M_k)>\tilde{\mathcal{V}}(\mathbb{S}^{n-1}),
\end{eqnarray*}
since $\tilde{\mathcal{V}}(\mathbb{S}^{n-1})<\tilde{\mathcal{V}}(\mathbb{S}^{n})$. Therefore $M_k$ must be compact and hence Einstein
for all $k$ large enough, by Theorem \ref{Thm:density} and by the proof of Lemma \ref{Lm:density_compact}, respectively. Furthermore, we have
\begin{eqnarray*}
|Rm_k|\leq C(n)R_k=C(n)\frac{n}{2}, & \operatorname{Vol}(g_k)=e^{\frac{n}{2}}(4\pi)^{\frac{n}{2}}\tilde{\mathcal{V}}(M_k)\geq c(n)>0.
\end{eqnarray*}
After possibly passing to a subsequence, $(M_k,g_k)$ will converge to an Einstein manifold having the same volume as $\mathbb{S}^n$, hence must be the standard $\mathbb{S}^n$ by the Bishop-Gromov volume comparison theorem. It follows that $g_k$ has strictly positive curvature operator and $M_k$ is diffeomorphic to $\mathbb{S}^n$ for all $k$ large enough because of the smooth convergence. But this is a contradiction, since the only Einstein manifold with strictly positive curvature operator that is diffeomorphic to $\mathbb{S}^n$ must be the standard $\mathbb{S}^n$, according to B\"ohm-Wilking \cite{bohm2008manifolds}.
\end{proof}

\section{On $\kappa$-solutions satisfying Assumption A}

\subsection{Preliminaries of Perelman's $\mathcal{L}$-geometry}
In this subsection we collect some well-known results concerning Perelman's $\mathcal{L}$-geometry. Let $(M,g(\tau))$ be a complete backward Ricci flow and let $(x_0,\tau_0)$ be a point in space-time. The \emph{reduced distance} based at $(x_0,\tau_0)$ and evaluated at $(x,\tau)$, where $\tau>\tau_0$, is defined by
\begin{eqnarray}\label{eq:l}
l_{(x_0,\tau_0)}(x,\tau)&=&\frac{1}{2\sqrt{\tau-\tau_0}}\inf_{\gamma}\mathcal{L}(\gamma)
\\\nonumber
&:=&\frac{1}{2\sqrt{\tau-\tau_0}}\inf_{\gamma}\int_{\tau_0}^\tau\sqrt{s-\tau_0}\big(R(\gamma(s),s)+|\dot{\gamma}(s)|_{g(s)}^2\big)ds,
\end{eqnarray}
where the $\inf$ is taken among all piecewise smooth curves $\gamma:[\tau_0,\tau]\rightarrow M$ satisfying $\gamma(\tau_0)=x_0$ and $\gamma(\tau)=x$. A minimizer of the $\mathcal{L}$-functional is called a minimizing $\mathcal{L}$-geodesic. The maximum principle implies that 
\begin{eqnarray}\label{eq:l_min}
\inf_M l_{(x_0,\tau_0)}(\cdot,\tau)\leq\frac{n}{2},
\end{eqnarray}
for all $\tau>\tau_0$. The \emph{reduced volume} based at $(x_0,\tau_0)$ and evaluated at $\tau>\tau_0$, is defined by
\begin{eqnarray*}
\mathcal{V}_{(x_0,\tau_0)}(\tau)=\frac{1}{(4\pi(\tau-\tau_0))^{\frac{n}{2}}}\int_Me^{-l(\cdot,\tau)}d\mu_{ g_\tau}(\cdot).
\end{eqnarray*}
It is known that
\begin{eqnarray*}
\lim_{\tau\rightarrow\tau_0+}\mathcal{V}_{(x_0,\tau_0)}(\tau)=1
\end{eqnarray*}
and the most significant property of the reduced volume is its monotonicity
\begin{eqnarray*}
\frac{d}{d\tau}\mathcal{V}_{(x_0,\tau_0)}(\tau)\leq 0.
\end{eqnarray*}
If the backward Ricci flow is eternal, then we define the \emph{asymptotic reduced volume} as
\begin{eqnarray}\label{eq:asymptotic_reduced_volume}
\tilde{\mathcal{V}}(x_0,\tau_0)=\lim_{\tau\rightarrow\infty}\mathcal{V}_{(x_0,\tau_0)}(\tau).
\end{eqnarray}
Since $0<\mathcal{V}_{(x_0,\tau_0)}(\tau)\leq 1$ and since $\mathcal{V}_{(x_0,\tau_0)}(\tau)$ is decreasing in $\tau$, we have that the above limit always exists, and is a function of the base point $(x_0,\tau_0)$.

All the important inequalities for the reduced distance are obtained by arguments concerning a (unique) minimizing $\mathcal{L}$-geodesic. In the literature there are several proofs for the existence of minimizing $\mathcal{L}$-geodesics connecting any two points assuming bounded curvature. Since we need to work in cases with possibly unbounded curvature, we establish the following lemma to justify our scenario.

\begin{Lemma}\label{Lm:L_minimizer}
Let $(M,g(\tau))_{\tau\in[0,T]}$ be a smooth backward Ricci flow with $\Ric\geq 0$. Let $x_0$ and $x\in M$ be any two points. Let $0<\tau\leq T$. Then there exists a minimizing $\mathcal{L}$-geodesic connecting $(x_0,0)$ and $(x,\tau)$. 
\end{Lemma}

\begin{proof}
Let $\gamma_k:[0,\tau]\rightarrow M$ satisfying $\gamma_k(0)=x_0$ and $\gamma_k(\tau)=x$ be a sequence of space-time curves such that $\mathcal{L}(\gamma_k)\rightarrow\inf_\gamma\mathcal{L}(\gamma)$. If we can prove that each $\gamma_k$ cannot escape a compact set, then the existance of minimizer is a standard result of calculus of variation. We may assume 
\begin{eqnarray*}
\mathcal{L}(\gamma_k)\leq 2C_0,
\end{eqnarray*}
where $C_0$ can be taken as $\mathcal{L}(\gamma_0)$ with $\gamma_0:[0,\tau]\rightarrow M$ being the minimizing geodesic connecting $x_0$ and $x$ with respect to the metric $g(0)$. In particular, $C_0$ can be bounded in terms of $\tau$, $D$, and the curvature bound in $\displaystyle B_{g(0)}(x_0,D)\times[0,\tau]$, where $D=d_0(x_0,x)$. Then we have
\begin{eqnarray*}
2C_0&\geq&\int_0^\tau\sqrt{s}\big(R(\gamma_k(s),s)+|\dot{\gamma_k}(s)|^2_{g(s)}\big)ds
\\
&\geq&\int_0^\tau\sqrt{s}|\dot{\gamma_k}(s)|^2_{g(0)}ds=\frac{1}{2}\int_0^{\sqrt{\tau}}|\dot{\bar{\gamma}}_k(\sigma)|^2_{g(0)}d\sigma,
\end{eqnarray*}
where we have used $R\geq0$, $\frac{\partial}{\partial s}g(s)=2 \Ric\geq 0$, and applied the change of variable $\bar{\gamma}_k(\sigma)=\gamma_k(\sigma^2)$. It follows that any minimizing geodesic $\eta:[0,\sigma]\rightarrow M$ with respect to $g(0)$ connecting $x_0$ and $\bar{\gamma}_k(\sigma)$ has energy smaller than $4C_0$. Hence
\begin{eqnarray*}
\frac{d_0(x_0,\bar{\gamma}_k(\sigma))^2}{\sigma}\leq 4C_0
\end{eqnarray*}
and
\begin{eqnarray*}
\gamma_k\subset B_{g(0)}\left(x_0,2\sqrt{C_0\tau^{\frac{1}{2}}}\right).
\end{eqnarray*}
This completes the proof.
\end{proof}

\emph{Remark.} Note that the proof of Lemma \ref{Lm:L_minimizer}  implies that every $\mathcal{L}$-minimizer is also contained in $B_{g(0)}\left(x_0,2\sqrt{C_0\tau^{\frac{1}{2}}}\right)$. By applying Corollary 6.67 in \cite{morgan2007ricci} to a compact exhaustion of $M$, one obtains that the $\mathcal{L}$-cut locus has zero measure. Furthermore, one may also prove the above Lemma assuming only a lower bound for the Ricci curvature.
\\

\emph{Remark.} By carefully checking the proof of most of the results stated below, and taking into account Lemma \ref{Lm:L_minimizer} and the above remark, one sees that one does not need to require the ancient solution to be a $\kappa$-solution, but it suffices to assume that the ancient solution has nonnegative curvature operator, is $\kappa$-noncollapsed, and that Hamilton's trace Harnack \cite{hamilton1993harnack} holds; we will always make it clear for which theorems this is true. This is to say, one may replace the curvature boundedness condition by the validity of Hamilton's trace Harnack---a differential inequality that is preserved under the Cheeger-Gromov-Hamilton convergence.
\\

From this point on, we consider only $\kappa$-solutions. By using the monotonicity of the reduced volume, Perelman proved the existence of asymptotic shrinkers for $\kappa$-solutions. The detailed proof of the following theorem can be found in \cite{morgan2007ricci}. Theorem \ref{Thm:asymptotic_shrinker} and Corollary \ref{Coro:shrinking_sphere} still hold if one replaces the curvature boundedness by the validity of Hamilton's trace Harnack.

\begin{Theorem}[Perelman, \cite{perelman2002entropy}] \label{Thm:asymptotic_shrinker}
Let $(M,g(\tau))_{\tau\in[0,\infty)}$ be a $\kappa$-solution, where $\tau$ stands for backward time. Let $l_{(x_0,\tau_0)}(x,\tau)$ be the reduced distance based at $(x_0,\tau_0)$. Let $\{(x_k,\tau_k)\}_{k=1}^\infty$ be a sequence of space-time points, such that $\tau_i\nearrow\infty$ and $\limsup_{k\rightarrow\infty}l_{(x_0,\tau_0)}(x_k,\tau_k)<\infty$. Then $\{(M,g_k(\tau),(x_k,1))_{\tau\in[\frac{1}{2},1]}\}_{k=1}^\infty$ converges, after possibly passing to a subsequence, to (the canonical form of) a nonflat shrinking gradient Ricci soliton, called the \emph{asymptotic shrinker based at $(x_0,\tau_0)$}, where $g_k(\tau)$ is the scaled flow $\displaystyle \frac{1}{\tau_k-\tau_0}g\left(\tau(\tau_k-\tau_0)+\tau_0\right)$. Furthermore, the Gaussian density of the asymptotic shrinker is the same as the asymptotic reduced volume $\tilde{\mathcal{V}}(x_0,\tau_0)$.
\end{Theorem}

As in \cite{perelman2002entropy}, we have the following observation.

\begin{Corollary}\label{Coro:shrinking_sphere}
Let $(M,g(\tau))_{\tau\in[0,\infty)}$ be a $\kappa$-solution. If its asymptotic shrinker is the standard sphere $\mathbb{S}^n$, then $(M^n,g)$ is the standard shrinking sphere.
\end{Corollary}

It is known that on a $\kappa$-solution, the reduced distance has exactly quadratic growth, and that the curvature can be controlled by the reduced distance. One may refer to \cite{morgan2007ricci} for the following two lemmas, especially Lemma 9.25 in \cite{morgan2007ricci} for the lower bound in (\ref{eq:l_quadratic}). Note that the following two lemmas still hold if we replace the curvature boundedness condition by Hamilton's trace Harnack.
\begin{Lemma}\label{Lm:l_quadratic}
There exist constants $c>0$ and $C<\infty$, depending only on the dimension $n$, such that the following holds. Let $(M,g(\tau))_{\tau\in[0,\infty)}$ be a $\kappa$-solution, where $\tau$ stands for the backward time. Then for any $x_1$ and $x_2\in M$ and $\tau>0$, we have
\begin{eqnarray}\label{eq:l_quadratic}
-l(x_1,\tau)-C+c\frac{d^2_\tau(x_1,x_2)}{\tau}\leq l(x_2,\tau)\leq l(x_1,\tau)+C+C\frac{d^2_\tau(x_1,x_2)}{\tau},
\end{eqnarray}
where $l$ is the reduced distance based at some point $(x_0,0)$.
\end{Lemma}

\begin{Lemma}\label{Lm:R<l}
There exists a constant $C<\infty$, depending only on the dimension $n$, such that the following holds. Let $(M,g(\tau))_{\tau\in[0,\infty)}$ be a $\kappa$-solution, where $\tau$ stands for the backward time. Then we have
\begin{eqnarray*}
|\nabla l|^2+R&\leq&\frac{Cl}{\tau},
\\
\left|\frac{\partial}{\partial \tau}l\right|&\leq&\frac{Cl}{\tau},
\end{eqnarray*}
where $l$ is the reduced distance based at some point $(x_0,0)$, and these inequalities are understood in the barrier sense or in the sense of distribution.
\end{Lemma}

To conclude this subsection, we summarize Perelman's precompactness in the following Theorem.

\begin{Theorem}\label{Thm:precompactness}
Let $\{(M_k,g_k(\tau),(x_k,0))\}_{k=1}^\infty$ be a sequence of $\kappa$-noncollapsed ancient solutions with nonnegative curvature operator satisfying $R_k(x_k,0)=1$, where $\tau$ stands for the backward time. Furthermore assume that on each $(M_k,g_k)$ Hamilton's trace Harnack holds. Let $l_k$ and $\mathcal{V}_k$ be the reduced distance and reduced volume based at $(x_k,0)$, respectively. Then after passing to a subsequence, $\{(M_k,g_k(\tau),(x_k,0))\}_{k=1}^\infty$ converges in the Cheeger-Gromov-Hamilton sense to a $\kappa$-noncollapsed ancient solution $(M_\infty,g_\infty,(x_\infty,0))$ with nonnegative curvature operator, on which Hamilton's trace Harnack holds. Furthermore, $l_k\rightarrow l_\infty$ in $C_{loc}^{\alpha}$ or weak $*W^{1,2}_{loc}$ sense, and $\mathcal{V}_k(\tau)\rightarrow\mathcal{V}_\infty(\tau)$ for every $\tau>0$, where $l_\infty$ and $\mathcal{V}_\infty$ are the reduced distance and reduced volume based at $(x_\infty,0)$, respectively.
\end{Theorem}

\begin{proof}
The convergence of ancient solutions is a consequence of Perelman's bounded curvature at bounded distance theorem and the $\kappa$-noncollapsing assumption. Note that bounded curvature at bounded distance follows from Corollary 11.5 in \cite{perelman2002entropy} and Hamilton's trace Harnack, and it does not require curvature boundedness. From Lemma \ref{Lm:l_quadratic} and Lemma \ref{Lm:R<l} we have that $l_k$ converges to a function $l_\infty$ on $M_\infty\times(0,\infty)$. To see $l_\infty$ is the reduced distance from $(x_\infty,0)$, we refer the readers to Lemma 7.66 in \cite{chow2007ricci}. Note that even if Lemma 7.66 requires uniform boundedness of curvature, since it is a local argument, we may still justify its proof as follows. Let us fix an arbitrary $\tau>0$ and $y\in M_\infty$. Let $y_k\in M_k$ be such that $y_k\rightarrow y$. By the first remark after Lemma \ref{Lm:L_minimizer}, we have that all $\mathcal{L}$-minimizers connecting $(x_k,0)$ and $(y_k,\tau)$, as well as the $\mathcal{L}$-minimizer connecting $(x_\infty,0)$ and $(y,\tau)$, are contained in compact sets with sizes uniformly bounded. Hence one may directly apply the proof of Lemma 7.66 in \cite{chow2007ricci} here.

Now we proceed to prove the convergence of reduced volume. By the smooth convergence of the Ricci flow, we can find $K<\infty$, depending on $\tau$, such that
\begin{eqnarray*}
\sup_k\sup_{s\in[0,\tau]}R_k(x_k,s)\leq K.
\end{eqnarray*}
Hence we have
\begin{eqnarray*}
l_k(x_k,\tau)\leq\frac{1}{2\sqrt{\tau}}\int_0^{\tau}\sqrt{s}Kds=\frac{1}{3}K\tau.
\end{eqnarray*}
By Lemma \ref{Lm:l_quadratic} we have that
\begin{eqnarray*}
c_1d^2_{g_k(\tau)}(x_k,x)-C_1\leq l_k(x_k,\tau)\leq C_1d^2_{g_k(\tau)}(x_k,x)+C_1,
\end{eqnarray*}
where $c_1$ and $C_1$ are positive constants depending on $\tau$ but are independent of $k$; this same estimate also holds for $l_\infty$. In other words if $\tau$ is fixed, then $l_k$ has uniformly quadratic growth centered at $x_k$. Therefore, in the definition of reduced volume, the integral outside a compact set is negligible. To wit, for any $\varepsilon>0$, we can find $A<\infty$, such that
\begin{eqnarray*}
\mathcal{V}_k(\tau)-\varepsilon\leq\int_{B_{g_k(\tau)}(x_k,A)}e^{-l_k(\cdot,\tau)}dg_k(\cdot,\tau)\leq \mathcal{V}_k(\tau),
\end{eqnarray*}
for all $k$. The conclusion follows from first taking $k\rightarrow\infty$, then $A\rightarrow\infty$, and finally $\varepsilon\rightarrow0$.
\end{proof}

\subsection{Necks at space infinity}

In this subsection, we show that every noncompact $\kappa$-solution satisfying Assumption A must be asymptotically cylindrical at space infinity. Note that Assumption A is a condition at time negative infinity. We first observe that Assumption A implies \emph{every} asymptotic shrinker is the standard cylinder. The idea is that the Gaussian density of a noncompact shrinker with nonnegative curvature operator cannot lie between that of a standard cylinder and $1$, unless it is a standard sphere; see Theorem \ref{Thm:density}. This idea was also implemented by the second author in \cite{zhangrigidity}. The following Lemma of Yokota \cite{yokota2009perelman} enables us to compare the asymptotic reduced volume based at different points.

\begin{Lemma}[Lemma 3.1 in \cite{yokota2009perelman}]\label{Lm:V_compare}
Let $(M,g(\tau))_{\tau\in[0,\infty)}$ be an ancient Ricci flow with Ricci curvature bounded from below, where $\tau$ stands for the backward time. Let $\tau_1>\tau_2$. Then we have that for any $x_1$ and $x_2\in M$,
\begin{eqnarray*}
\tilde{\mathcal{V}}(x_1,\tau_1)\geq\tilde{\mathcal{V}}(x_2,\tau_2).
\end{eqnarray*}
\end{Lemma}

\begin{Lemma}\label{Lm:V_lower_bound}
Let $(M,g(\tau))_{\tau\in[0,\infty)}$ be a $\kappa$-solution satisfying Assumption A, where $\tau$ stands for backward time. Let  $(x_1,\tau_1)$ be any space-time point in $M\times [0,\infty)$. Then the asymptotic shrinker based at $(x_1,\tau_1)$ is the standard cylinder $\mathbb{S}^{n-1}\times\mathbb{R}$. In particular, we have
\begin{eqnarray*}
\mathcal{V}_{(x_1,\tau_1)}(\tau)\geq\tilde{\mathcal{V}}(\mathbb{S}^{n-1}\times\mathbb{R})
\end{eqnarray*}
for all $\tau>\tau_1$.
\end{Lemma}
\begin{proof}
Let $(x_0,\tau_0)$, where $\tau_0<0$, be a space-time point such that the asymptotic shrinker based at $(x_0,\tau_0)$ is the standard cylinder (see the remark after the definition of Assumption A). Since $\tau_1\geq 0>\tau_0$, we have by Lemma \ref{Lm:V_compare} that
\begin{eqnarray*}
\tilde{\mathcal{V}}(x_1,\tau_1)\geq\tilde{\mathcal{V}}(x_0,\tau_0)=\tilde{\mathcal{V}}(\mathbb{S}^{n-1}\times\mathbb{R}).
\end{eqnarray*}
By Theorem \ref{Thm:density}, it follows that the asymptotic shrinker based at $(x_1,\tau_1)$ can only be either $\mathbb{S}^{n-1}\times\mathbb{R}$ or $\mathbb{S}^n$. However, the latter case cannot occur, since otherwise $(M^n,g(\tau))$ would be the standard shrinking sphere by Corollary \ref{Coro:shrinking_sphere}, and consequently the asymptotic shrinker based at $(x_0,\tau_0)$ is also the standard sphere.
\end{proof}

Next, we show that all $\kappa$-solutions satisfying Assumption A must be asymptotic cylindrical. Note that Theorem \ref{Thm:asymptotic_cylindrical} follows from Lemma \ref{Lm:V_lower_bound} and Proposition \ref{prop:asymptotic_cylinder} below.

\begin{Proposition}\label{prop:asymptotic_cylinder}
Let $(M^n,g(\tau))_{\tau\in[0,\infty)}$ be a noncompact $\kappa$-noncollapsed ancient solution to the Ricci flow, where $\tau$ stands for the backward time. Assume that $g$ is nonflat with nonnegative curvature operator and that Hamilton's trace Harnack holds on $(M^n,g)$. Furthermore, assume that
\begin{eqnarray*}
\mathcal{V}_{(x_0,0)}(\tau)\geq\tilde{\mathcal{V}}(\mathbb{S}^{n-1}\times\mathbb{R})
\end{eqnarray*}
for all $x_0\in M^n$ and $\tau>0$. Let $\{x_k\}_{k=0}^\infty$ be an arbitrary sequence such that $d_0(x_0,x_k)\longrightarrow\infty$. Then $\{(M^n,g_k(\tau),(x_k,0))\}_{k=1}^\infty$ converges, after passing to a subsequence, to the standard shrinking cylinder $\mathbb{S}^{n-1}\times\mathbb{R}$. Here $g_k(\tau)$ is the scaled flow $Q_kg(\tau Q_k^{-1})$, where $Q_k=R(x_k,0)$.
\end{Proposition}

\begin{proof}
The convergence and splitting are classical results of Perelman \cite{perelman2002entropy}. Let $\big(N^{n-1}\times\mathbb{R},\bar{g}(\tau)+dz\otimes dz,((y_0,0),0)\big)$ be the limit, where $\big(N^{n-1},\bar{g},(y_0,0)\big)$ is an $n-1$-dimensional nonflat $\kappa$-noncollapsed ancient solution on which Hamilton's trace Harnack holds. Let $\mathcal{V}_k$ be the reduced volume of $g_k$ based at $(x_k,0)$. Since the reduced volume is invariant with respect to parabolic scaling, we have that
\begin{eqnarray*}
\mathcal{V}_k(\tau)\geq\tilde{\mathcal{V}}(\mathbb{S}^{n-1}\times\mathbb{R})=\tilde{\mathcal{V}}(\mathbb{S}^{n-1}),
\end{eqnarray*}
for every $k$ and for all $\tau>0$. By Theorem \ref{Thm:precompactness} we have
\begin{eqnarray*}
\mathcal{V}^{\bar{g}}_{(y_0,0)}(\tau)=\mathcal{V}^{\bar{g}+dz\otimes dz}_{\big((y_0,0),0\big)}(\tau)=\lim_{k\rightarrow\infty}\mathcal{V}_k(\tau)\geq\tilde{\mathcal{V}}(\mathbb{S}^{n-1}),
\end{eqnarray*}
for all $\tau>0$. By Theorem \ref{Thm:asymptotic_shrinker} we have that the asymptotic shrinker of $(N^{n-1},\bar{g})$ is nonflat with Gaussian density no less than that of $\mathbb{S}^{n-1}$, it must then be the standard sphere by Theorem \ref{Thm:density}. Hence $(N^{n-1},\bar{g})$ can be nothing but the standard shrinking sphere by Corollary \ref{Coro:shrinking_sphere}; this completes the proof.
\end{proof}

The following Corollary also holds when one replaces the curvature boundedness by Hamilton's trace Harnack.

\begin{Corollary}
If a $\kappa$-solution satisfying Assumption A splits, it must be the standard shrinking cylinder.
\end{Corollary}

\subsection{The $\kappa$-compactness theorem}

With the preparations in the last two subsections, we prove the $\kappa$-compactness theorem (Theorem \ref{Thm:kappa_compactness}) in this subsection. As Perelman \cite{perelman2002entropy} has already established bounded curvature at bounded distance, the main point is to prove that the limit has bounded curvature. Our idea of proof is the same as Perelman's in dimension three---since the limit is asymptotically cylindrical at space infinity, unbounded curvature implies $\varepsilon$-necks of arbitrarily small radii, which cannot happen in a manifold with nonnegative curvature operator.

\begin{proof}[Proof of Theorem \ref{Thm:kappa_compactness}]
Let $\{(M_k,g_k(\tau),(x_k,0))_{\tau\in[0,\infty)}\}_{k=1}^\infty$ be a sequence of $\kappa$-solutions satisfying Assumption A. Furthermore, let us assume that $R_k(x_k,0)=1$ and let $(M_\infty,g_\infty(\tau),(x_\infty,0))_{\tau\in[0,\infty)}$ be a nonflat $\kappa$-noncollapsed ancient solution with nonnegative curvature operator, on which Hamilton's trace Harnack holds, such that
\begin{eqnarray*}
\{(M_k,g_k(\tau),(x_k,0))_{\tau\in[0,\infty)}\}_{k=1}^\infty\rightarrow(M_\infty,g_\infty(\tau),(x_\infty,0))_{\tau\in[0,\infty)}
\end{eqnarray*}
in the Cheeger-Gromov-Hamilton sense. Note that the existence of such $(M_\infty,g_\infty(\tau),(x_\infty,0))_{\tau\in[0,\infty)}$ arises from Theorem \ref{Thm:precompactness}. Now we proceed to show that $(M_\infty,g_\infty(\tau),(x_\infty,0))_{\tau\in[0,\infty)}$ indeed has all the properties claimed in the statement of Theorem \ref{Thm:kappa_compactness}.
\\

\noindent\textbf{Claim.} For any $x_1\in M_\infty$, we have
\begin{eqnarray*}
\mathcal{V}_{(x_1,0)}(\tau)\geq \tilde{\mathcal{V}}(\mathbb{S}^{n-1}\times\mathbb{R}),
\end{eqnarray*}
for all $\tau>0$.

\begin{proof}[Proof of the claim]
Let $\bar{x}_k\in M_k$ be such that $(\bar{x}_k,0)\rightarrow (x_1,0)$. In particular, we have
\begin{eqnarray}\label{eq:R_converge}
R_k(\bar{x}_k,0)\rightarrow R_\infty(x_1,0)>0.
\end{eqnarray}
Then we have the following convergence in the Cheeger-Gromov-Hamilton sense.
\begin{eqnarray*}
\{(M_k,g_k(\tau),(\bar{x}_k,0))_{\tau\in[0,\infty)}\}_{k=1}^\infty\rightarrow(M_\infty,g_\infty(\tau),(x_1,0))_{\tau\in[0,\infty)}.
\end{eqnarray*}
By Lemma \ref{Lm:V_lower_bound} we have
\begin{eqnarray*}
\mathcal{V}_{(\bar{x}_k,0)}(\tau)\geq\tilde{\mathcal{V}}(\mathbb{S}^{n-1}\times\mathbb{R})
\end{eqnarray*}
for all $\tau>0$ and for each $k$. Therefore by Theorem \ref{Thm:precompactness} we have that
\begin{eqnarray*}
\mathcal{V}_{(x_1,0)}(\tau)=\lim_{k\rightarrow\infty}\mathcal{V}_{(\bar{x}_k,0)}(\tau)\geq\tilde{\mathcal{V}}(\mathbb{S}^{n-1}\times\mathbb{R})
\end{eqnarray*}
for all $\tau>0$. Note that in Theorem \ref{Thm:precompactness} we assumed that the scalar curvatures at the base points are exactly one, which might not be true for $\bar{x}_k$'s. However, because of (\ref{eq:R_converge}), this makes a difference only in the scaling factors, which does not affect lower bounds for the reduced volumes. This completes the proof of the claim.
\end{proof}

We continue the proof of the theorem. To see that $(M_\infty,g_\infty)_{\tau\in[0,\infty)}$ has bounded curvature, we need only to show that $g_\infty(0)$ has bounded curvature, since Hamilton's trace Harnack implies that $\frac{\partial}{\partial\tau}R_\infty\leq 0$. Suppose this is not true, we can find a sequence $y_k$ such that $d_{g_\infty(0)}(y_k,x_0)\rightarrow\infty$ and $R_\infty(y_k,0)\rightarrow\infty$. By Proposition \ref{prop:asymptotic_cylinder}, we can take a scaled limit along $(y_k,0)$ and obtain a shrinking cylinder. It follows that $(M_\infty,g_\infty(0))$ contains $\varepsilon$-necks of arbitrary small scales; this is a contradiction (c.f. Proposition 2.2 in \cite{chen2006ricci}).

By the above claim, we have that the asymptotic shrinker based at any point in $M_\infty\times[0,\infty)$ is either the cylinder or the sphere. It is clear that the last statement of Theorem \ref{Thm:kappa_compactness} holds.

\end{proof}

\emph{Remark.} To see why the limit can possibly be the shrinking sphere, one may take Perelman's ancient solution which he constructed in 1.4 of \cite{perelman2003ricci}. Fix a point in space and choose a sequence of times approaching the singular time, then the blow-up limit along the space-time sequence is the shrinking sphere. However, Perelman's solutions has $\mathbb{S}^{n-1}\times\mathbb{R}$ as the asymptotic shrinker.

\subsection{The geometry of $\kappa$-solutions satisfying Assumption A}

In this subsection, we collect some consequences of the $\kappa$-compactness theorem proved in the last subsection. Let us consider a noncompact $\kappa$-solution satisfying Assumption A with strictly positive curvature operator $(M^n,g(t))_{t\in(-\infty,0]}$, then $M$ is diffeomorphic to $\mathbb{R}^n$. Let $M_{\varepsilon}$ denote all the points that are not centers of $\varepsilon$-necks at some certain time, say $t=0$. We follow Corollary 48.1 in \cite{kleiner2008notes} to outline the proof of Theorem \ref{Thm:Size_of_cap}.

\begin{proof} [Proof of Theorem \ref{Thm:Size_of_cap} (sketch)]
Indeed, $M_\varepsilon$ is compact by Proposition \ref{prop:asymptotic_cylinder}. We need only to prove the diameter bound for $M_\varepsilon$, then all the rest follows from Perelman's bounded curvature at bounded distance theorem for $\kappa$-solutions. In the following argument we always assume $t=0$.

\noindent\textbf{Claim:} There exists $\alpha>0$ depending on $\kappa$ and $\varepsilon$, such that the following holds. Assume that $x$ and $y\in M_\varepsilon$ satisfy $R(x)d^2(x,y)>\alpha$. Then for any $z$, one of the following three holds
\begin{enumerate}[(1)]
\item$R(x)d^2(x,z)<\alpha$;
\item $R(y)d^2(y,z)<\alpha$; \item$R(z)d^2(z,\overline{xy})<\alpha$ and $z\notin M_\varepsilon$.
\end{enumerate}

\begin{proof}[Proof of the claim]
Assume the claim is not true, then for some $\varepsilon$ and $\kappa$ we can find a contradicting sequence: $\kappa$-solutions $\{(M_k,g_k(t))_{t\in(-\infty,0]}\}_{k=1}^\infty$ satisfying Assumption A, $x_k$, $y_k\in(M_k)_\varepsilon$, and $z_k\in M_k$, such that
\begin{eqnarray*}
R_k(x_k)d^2(x_k,y_k)\rightarrow\infty, & R_k(x_k)d^2(x_k,z_k)\rightarrow\infty, & R_k(y_k)d^2(y_k,z_k)\rightarrow\infty.
\end{eqnarray*}
Let $z_k'$ be the point on $\overline{x_ky_k}$ such that $d(z_k,z_k')=d(z_k,\overline{x_ky_k})$. We show that $R_k(x_k)d^2(z_k',x_k)\rightarrow\infty$. If not, then we can take a scaled limit along $x_k$, and obtain a $\kappa$-solution satisfying the properties stated in the conclusion of Theorem \ref{Thm:kappa_compactness}. Let $z_\infty'$ be the limit of $z_k'$ and $x_\infty$ the limit of $x_k$. Since $R_k(x_k)d^2(x_k,y_k)\rightarrow\infty$ and $R_k(x_k)d^2(x_k,z_k)\rightarrow\infty$, we have that $\overline{x_ky_k}$ converges to a ray $\overline{x_\infty\xi}$, and $\overline{z_k'z_k}$ converges to a ray $\overline{z_\infty'\eta}$. Since for every point on $\overline{z_\infty'\eta}$, its distance to $\overline{x_\infty\xi}$ is fulfilled by the ray $\overline{z_\infty'\eta}$, we have that the Tits angle $\angle_T(\eta z_\infty'\xi)\geq\frac{\pi}{2}$. On the other hand, by Proposition \ref{prop:asymptotic_cylinder}, along any sequence of points on $\overline{x_\infty\xi}$ going off to space infinity, the scaled limit must be the standard cylinder, this gives us an end along $\overline{x_\infty\xi}$, being an $\varepsilon$-tube. But this cannot be the only end, because otherwise the ray $\overline{z_\infty'\eta}$ will also go to infinity along the same end. If this happens, to ensure that $\angle_T(\zeta z_\infty'\xi)\geq\frac{\pi}{2}$, the radius of each central sphere of all $\varepsilon$-necks on this end must be at least proportional to the distance from $z_\infty'$, this gives positive \emph{asymptotic volume ratio}---something that cannot happen on a $\kappa$-solution (c.f. 11.4 in \cite{perelman2002entropy}). Since the limit has two ends, it splits as a line and an $(n-1)$-dimensional $\kappa$-solution, and by Theorem \ref{Thm:kappa_compactness} this is the standard cylinder, but $x_\infty$ is not the center of an $\varepsilon$-neck; this is a contradiction. By the same reason we have $R_k(y_k)d^2(z_k',y_k)\rightarrow\infty$. Now taking a scaled limit along $z_k'$, we obtain a $\kappa$-solution containing a line, it must be the cylinder by Theorem \ref{Thm:kappa_compactness}. It follows that when $k$ is large enough, $z_k$ is close to $z_k'$ and is in the neck-like region; this is a contradiction.
\end{proof}

We continue the proof of the theorem. By the claim, we have that if there are two points in $M_\varepsilon$ that are too far from each other, then the manifold $M$ must be compact. To be more precise, we observe that if case (3) in the above claim occurs, then, taking into account Perelman's bounded curvature at bounded distance (as a consequence of Theorem \ref{Thm:kappa_compactness}), there exists $\xi$ on $\overline{xy}$ such that 
\begin{eqnarray*}
R(\xi)d^2(z,\xi)\leq C,
\end{eqnarray*}
where $C$ depends on $\kappa$ and $\alpha$. Hence, if every $R(x)d^2(x,y)>\alpha$, the above claim implies that for any $z\in M$, it holds that 
\begin{eqnarray*}
d^2(z,\overline{xy})\leq C_1,
\end{eqnarray*}
where $C_1$ depends on $\kappa$, $\alpha$, $\varepsilon$, and $\inf_{\xi\in\overline{xy}}R(\xi)$. It follows that $M$ is compact; this is a contradiction.

\end{proof}

To conclude this section we state the following neck stability result due to Kleiner-Lott \cite{kleiner2014singular}. 

\begin{Theorem}[Kleiner-Lott \cite{kleiner2014singular}]\label{Thm:neck_stability}
For any positive numbers $\varepsilon_1$ and $\varepsilon_2$ small enough, there exists $0>T>-\infty$, depending also on $\kappa$, such that the following holds. Let $(M,g(t),(x_0,0))_{t\in(-\infty,0]}$ be a $\kappa$-solution satisfying Assumption A. Assume that $R(x_0,0)=1$ and that $(x_0,0)$ is the center of an $\varepsilon_1$-neck. Then for all $t<T$, $(x_0,t)$ is the center of an $\varepsilon_2$-neck.
\end{Theorem}

\begin{proof}
When Assumption A is made, the only possible asymptotic shrinkers are the standard cylinder, hence our scenario has no difference from three-dimensional case. Furthermore, the proof of Kleiner-Lott does not depend on any result that is valid only for three-dimensional geometry (except for the $(\mathbb{S}^2\times\mathbb{R})/\mathbb{Z}_2$ case, which we will never encounter), one may follow the proof of Kleiner-Lott line-by-line to obtain this theorem. One may also refer to \cite{zhangrigidity} for a simpler proof. Note that the Type I assumption in \cite{zhangrigidity} was used to deal with the possibility that the limit has unbounded curvature, so it is not necessary since we already have Theorem \ref{Thm:kappa_compactness}.
\end{proof}

\section{On the rotationally symmetric $\kappa$-solution}

In this section we proceed to prove Theorem \ref{Thm:symmetric}. Throughout this section we consider a $\kappa$-solution on $\mathbb{R}^n$ that is rotationally symmetric with respect to its origin $O$. As we will see in subsection 4.2, a rotationally symmetric $\kappa$-solution satisfies Assumption A. Given that we have established Theorem \ref{Thm:asymptotic_cylindrical}---\ref{Thm:Size_of_cap}, the proof is the same as the three-dimensional case in \cite{brendle2018ancient}, except that we need barriers in higher dimensions. In subsection 4.3 we generalize Brendle's barriers to higher dimensions. In subsection 4.4 we outline the main steps of the proof and the details are omitted.

\subsection{The ansatz}

We consider the evolving warped product $\displaystyle g(t)=\frac{1}{u(r,t)}dr\otimes dr+r^2g_{\mathbb{S}^{n-1}}$, then the Ricci curvature and the scalar curvature are
\begin{eqnarray*}
\Ric&=&-\frac{n-1}{2r}u^{-1}u_rdr\otimes dr+\big((n-2)(1-u)-\frac{1}{2}ru_r\big)g_{\mathbb{S}^{n-1}},
\\
R&=&\frac{n-1}{r^2}\big((n-2)(1-u)-ru_r\big).
\end{eqnarray*}
Letting $\displaystyle V=v\frac{\partial}{\partial r}$ with
\begin{eqnarray}\label{eq:def_v}
v&=&\frac{1}{r}\big((n-2)(1-u)-\frac{1}{2}ru_r\big),
\end{eqnarray}
we may compute
\begin{eqnarray*}
\Ric-\frac{1}{2}\mathcal{L}_Vg&=&\left(-\frac{n-1}{2r}u^{-1}u_r+\frac{1}{2}u^{-2}u_rv-u^{-1}v_r\right)dr\otimes dr
\\
&=&\frac{1}{2}u^{-2}\left(uu_{rr}-\frac{1}{2}(u_r)^2+(n-2)\frac{u_r}{r}-\frac{uu_r}{r}+\frac{2(n-2)}{r^2}u(1-u)\right)dr\otimes dr.
\end{eqnarray*}
It follows that the modified Ricci flow equation
\begin{eqnarray*}
\frac{\partial}{\partial t}g=-2\Ric+\mathcal{L}_Vg
\end{eqnarray*}
becomes the following one-dimensional parabolic equation
\begin{eqnarray}\label{eq:the_pde}
\frac{\partial}{\partial t}u=uu_{rr}-\frac{1}{2}(u_r)^2+(n-2)\frac{u_r}{r}-\frac{uu_r}{r}+\frac{2(n-2)}{r^2}u(1-u).
\end{eqnarray}
In the above equation, if the left-hand-side is $0$, then we have a steady gradient Ricci soliton.

The nonnegativity of the curvature operator implies the following conditions
\begin{eqnarray}
0<u\leq 1, & v\geq0, & u_r\leq 0.
\end{eqnarray}
Moreover, since the metric extends smoothly across the origin, we have
\begin{eqnarray*}
1-u(r,t)=O(r^2), & v(r,t)=O(r),
\end{eqnarray*}
when $r\rightarrow 0$.




\subsection{Validity of Assumption A}

In this subsection we show that a rotationally symmetric $\kappa$-solution $(M,g(t))_{t\in(-\infty,0]}$, of necessity, satisfies Assumption A. For the sake of convenience, we always use $O$ to denote the center of symmetry.

\begin{Lemma}\label{Lm:type_II}
Let $(M,g(\tau))_{\tau\in[0,\infty)}$ be a rotationally symmetric $\kappa$-solution with positive curvature operator, where $\tau$ stands for the backward time. We have $\limsup_{\tau\rightarrow\infty}\tau R(O,\tau)=\infty$. In particular, $g(\tau)$ is of Type II.
\end{Lemma}
\begin{proof}
Assume by contradiction that $\tau R(O,\tau)\leq C$, then
\begin{eqnarray*}
l_{(O,0)}(O,\tau)\leq\frac{1}{2\sqrt{\tau}}\int_0^\tau\sqrt{s}\frac{C}{s}\leq C.
\end{eqnarray*}
By Theorem \ref{Thm:asymptotic_shrinker}, for $\tau_k\nearrow\infty$ the scaled limit $(M,\tau^{-1}g(\tau\tau_k),(O,1))_{[\frac{1}{2},1]}$ converges to a nonflat asymptotic shrinker, which must also be rotationally symmetric with respect to its base point. There exists no such shrinker and this is a contradiction; see Kotschwar \cite{kotschwar2008rotationally}.
\end{proof}

Now let us take $\tau_k\nearrow\infty$ such that $\tau_kR(O,\tau_k)\nearrow\infty$. Let $x_k$ be such that $l_{(O,0)}(x_k,\tau_k)\leq\frac{n}{2}$. Such $x_k$'s exist because of (\ref{eq:l_min}). By Lemma \ref{Lm:l_quadratic} and Lemma \ref{Lm:R<l}, we have that
\begin{eqnarray*}
R(x,\tau_k)\leq C\left(\frac{1}{\tau_k}+\frac{d^2_{\tau_k}(x_k,x)}{\tau_k^2}\right),
\end{eqnarray*}
where $C$ is independent of $k$. Since  $\tau_kR(O,\tau_k)\nearrow\infty$, we must have $\frac{d^2_{\tau_k}(x_k,O)}{\tau_k}\nearrow\infty$, or in other words, the distance between $O$ and $x_k$ with respect to the scaled metric $\tau_k^{-1}g(\tau_k)$ must go to infinity. Hence, the asymptotic shrinker obtained along $(x_k,\tau_k)$ must contain a geodesic line, which arises from the geodesic ray emanated from $O$ and passing through $x_k$. Let the asymptotic shrinker be $(N^{n-1}\times\mathbb{R},\bar{g}+dz\otimes dz, (y_0,0))$, we proceed to show that $N$ is the standard sphere. Let $\{e^{(k)}_i\}_{i=1}^n$ be an orthonormal frame at $x_k$ with respect to the metric $\tau_k^{-1}g(\tau_k)$, where $e^{(k)}_n$ is along the radial direction. We then have $Ric_k(e^{(k)}_i,e^{(k)}_i)$ are all the same and $R_k(e^{(k)}_i,e^{(k)}_j,e^{(k)}_j,e^{(k)}_i)$ are all the same for $i\neq j$ and $i$, $j\neq n$. Let $\{e^\infty_i\}_{i=1}^n$ be an orthonormal frame on $N^{n-1}\times\mathbb{R}$ at $(y_0,0)$ such that $\{e^{(k)}_i\}_{i=1}^n\rightarrow\{e^\infty_i\}_{i=1}^n$. We then have that $Ric_\infty(e^\infty_i,e^\infty_i)$ are all the same and positive and that $R_\infty(e^\infty_i,e^\infty_j,e^\infty_j,e^\infty_i)$ are all the same and positive, for $i\neq j$ and $i$, $j\neq n$ (it is easy to see that all such components of $Ric_\infty$ are the same, but since the limit shrinker is nonflat and has nonnegative curvature operator, they must all be positive; similarly all such sectional curvatures are positive). This fact implies $e^\infty_n=\partial_z$, for if $Ric_\infty(e^\infty_n,e^\infty_n)>0$, then $Ric_\infty$ would have no zero eigenvalues, and hence $g_\infty$ cannot split off a line. As a consequence, the radial distance function $r$, after scaling, converges to $z$ locally uniformly.

According to the argument above, we have that $N^{n-1}$ has nonnegative curvature operator and has positive curvature operator at one point $y_0$. It then follows that $N^{n-1}$ is compact (c.f. \cite{munteanu2017positively}) and is a round space form (c.f. \cite{bohm2008manifolds}). Next we show $N^{n-1}$ must be the standard sphere. For the sake of convenience we write $\tau_k^{-1}g(\tau_k)=dr^2+\rho_k(r)^2g_{\mathbb{S}^{n-1}}$. Since
\begin{eqnarray*}
R_k(e^{(k)}_i,e^{(k)}_j,e^{(k)}_j,e^{(k)}_i)\rightarrow R_\infty(e^\infty_i,e^\infty_j,e^\infty_j,e^\infty_i):=K>0
\end{eqnarray*}
for all $i\neq j$ and $i$, $j\neq n$, we have that at $x_k$
\begin{eqnarray*}
\frac{1-\dot{\rho}_k^2}{\rho_k^2}\rightarrow K>0.
\end{eqnarray*}
Consequently $\rho_k^2$ is bounded from above by $2K^{-1}$ at $x_k$ for all $k$ large. Note that the convergence of the Ricci flows is in the locally smooth sense, we actually have $\rho_k^2$ is bounded by $2K^{-1}$ in larger and larger domains centered at $x_k$. It follows that the radii of the $\mathbb{S}^{n-1}$ factor, after scaling, is uniformly bounded, and one sees longer and longer necks $\mathbb{S}^{n-1}\times\mathbb{I}$ with bounded radius near $x_k$. Hence $N^{n-1}\times\mathbb{R}$ is diffeomorphic to $\mathbb{S}^{n-1}\times\mathbb{R}$, and consequently is isometric to the standard $\mathbb{S}^{n-1}\times\mathbb{R}$. Now we have the following Proposition.

\begin{Proposition}
A rotationally symmetric $\kappa$-solution satisfies Assumption A.
\end{Proposition}

\begin{Proposition}\label{Prop:asymptotic_bryant}
Let $(M,g(t), O)_{t\in(-\infty,0]}$ be a rotationally symmetric $\kappa$-solution. There exists a sequence of times $t_k\rightarrow-\infty$, such that $(M,g_k(t),O)_{t\in(-\infty,0]}$ converges to the Bryant soliton. Here $g_k(t)=Q_kg(t_k+tQ_k^{-1})$ and $Q_k=R(O,t_k)$.
\end{Proposition}

\begin{proof}
Since $g(t)$ is of Type II, we can apply Hamilton's dilation procedure \cite{hamilton1995formation} to choose a sequence of space-time points $(x_k,t_k)$, such that the limit of $(M,\bar{g}_k(t),x_k)_{t\in(-\infty,0]}$ is a steady soliton, here $\bar{g}_k(t)=\bar{Q}_kg(t\bar{Q}_k^{-1}+t_k)$ and $\bar{Q}_k=R(x_k,t_k)$.

We claim that $d_{\bar{g}_k(0)}(x_k,O)$ must be bounded. Suppose it is not, the limit contains a geodesic line that arises from the geodesic rays emanating from $O$ and passing through $x_k$, and hence splits. By Theorem \ref{Thm:kappa_compactness}, the limit Ricci steady is the standard shrinking cylinder; this is a contradiction.

Since $d_{\bar{g}_k(0)}(x_k,O)$ is bounded, by Perelman's bounded curvature at bounded distance theorem, the two different dilations $(M,\bar{g}_k(t),x_k)_{t\in(-\infty,0]}$ and $(M,g_k(t),O)_{t\in(-\infty,0]}$, where $g_k(t)=Q_kg(t_k+tQ_k^{-1})$ and $Q_k=R(O,t_k)$, are equivalent, since the ratios of their scaling factors are bounded from above and below, and the distances between their base points are bounded.

It follows that the limit of $(M,g_k(t),O)_{t\in(-\infty,0]}$ is also a steady soliton. Since it is rotationally symmetric, it must be the Bryant soliton.

\end{proof}

We let $R_{\max}(t)$ be the supremum of scalar curvature of $(M,g(t))$ and define
\begin{eqnarray*}
\mathcal{R}=\lim_{t\rightarrow-\infty}R_{\max}(t).
\end{eqnarray*}
Note that by Hamilton's trace Harnack we have $\frac{\partial R}{\partial t}\geq 0$, hence $R_{\max}(t)$ is increasing in $t$. It follows that the above limit always exists, but could possibly be $0$.

\begin{Corollary}
Assume $\mathcal{R}>0$. Let $\{t_i\}_{i=1}^\infty$ be an arbitrary sequence such that $t_i\rightarrow-\infty$. Then the (not rescaled) sequence $\{(M,g(t+t_i),O)_{t=(-\infty,0]}\}_{i=1}^\infty$ converges, after passing to a subsequence, to the Bryant soliton with maximum scalar curvature being equal to $\mathcal{R}$.
\end{Corollary}
\begin{proof}
By passing to a subsequence, we can assume
\begin{eqnarray*}
\mathcal{R}\leq R_{\max}(t_i+i)\leq \mathcal{R}+i^{-1}.
\end{eqnarray*}
Next, we choose $x_i$ such that
\begin{eqnarray*}
R(x_i,t_i)\geq R_{\max}(t_i)-i^{-1}\geq\mathcal{R}-i^{-1}.
\end{eqnarray*}
Obviously, the sequence $\{(M,g(t+t_i),(x_i,0))_{t\in(-\infty,i]}\}_{i=1}^\infty$ converges to a nonnegatively curved eternal solution $(M_\infty,g_\infty(t),(x_\infty,0))_{t\in(-\infty,+\infty)}$ with $R_\infty\leq\mathcal{R}$ everywhere and $R_\infty(x_\infty,0)=\mathcal{R}$. By Hamilton \cite{hamilton1993eternal}, $(M_\infty,g_\infty)$ is a steady Ricci soliton with maximum scalar curvature being $\mathcal{R}$. Furthermore, by the same argument as in the proof of Proposition \ref{Prop:asymptotic_bryant}, we have that $d_{t_i}(x_i,O)$ must be uniformly bounded, and hence $\{(M,g(t+t_i),(x_i,0))_{t\in(-\infty,i]}\}_{i=1}^\infty$ and $\{(M,g(t+t_i),(O,0))_{t\in(-\infty,i]}\}_{i=1}^\infty$ have the same limit.
\end{proof}

\subsection{Construction of the barriers}

In this subsection, we extend the barriers constructed in section 2 of \cite{brendle2018ancient} to all dimensions. We make use of the following half-complete rotationally symmetric steady soliton constructed in Proposition 2.2 of \cite{alexakis2015singular} (see also Bryant \cite{bryant2005ricci} for the case of dimension three).

\begin{Theorem}
There is a rotationally symmetric steady soliton singular at the tip and asymptotic to the Bryant soliton at infinity. If we write the metric as $\displaystyle g=\frac{1}{\varphi(r)}dr\otimes dr+r^2g_{\mathbb{S}^{n-1}}$, then $\varphi$ satisfies the equation
\begin{eqnarray*}
\varphi(r)\varphi''(r)-\frac{1}{2}(\varphi'(r))^2+(n-2)\frac{\varphi'(r)}{r}-\frac{\varphi(r)\varphi'(r)}{r}+\frac{2(n-2)}{r^2}\varphi(r)\big(1-\varphi(r)\big)=0.
\end{eqnarray*}
Furthermore,
\begin{eqnarray}
\varphi(r)&\rightarrow& +\infty \text{ as } r\rightarrow 0,
\\
\varphi(r)&=&(n-2)^2 r^{-2}-(n-5)(n-2)^3r^{-4}+O(r^{-6}), \text{ as } r \to \infty. \label{eq:asymptotes}
\end{eqnarray}
\end{Theorem}

\emph{Remark.} We sketch how to obtain the accurate asymptotic behavior as stated in (\ref{eq:asymptotes}). If we write the warped product as $g=ds^2+\psi(s)^2 g_{\mathbb{S}^{n-1}}$ and let
\begin{eqnarray*}
x=\psi^2, & y=(\psi')^2, & z=\psi\psi'',
\end{eqnarray*}
then in dimension $n$ the formulas (3.1) and (3.2) in \cite{bryant2005ricci} become
\begin{eqnarray}
0&=&xdy-zdx=2xydz-\big(z^2+2yz+(n-2)y^2-(n-2)z-(n-2)y\big)dx,\label{eq:bryant_1}
\\
Cxy&=&-(n-2)y^2-2yz+z^2-(n-3)(n-2)y-2(n-2)z+(n-2)^2, \label{eq:bryant_2}
\end{eqnarray}
where (\ref{eq:bryant_2}) is a first integral of (\ref{eq:bryant_1}). Note that $x=r^2$ and $y=\varphi$. As indicated by (2.14) and (2.15) in \cite{alexakis2015singular}, the singular soliton described in the above theorem corresponds an integral curve of (\ref{eq:bryant_1}) that approaches $(x,y,z)=(+\infty,0,0)$. By fixing a proper scale we may let $C=1$ in (\ref{eq:bryant_2}), and this gives the following formula that is equivalent to (\ref{eq:asymptotes})
\begin{eqnarray*}
y=(n-2)^2x^{-1}-(n-5)(n-2)^3x^{-2}+O(x^{-3}), \text{ as } x\rightarrow\infty.
\end{eqnarray*}


As in \cite{brendle2018ancient}, we fix $r_*>0$ such that $\varphi(r_*)=2$ and consider a smooth function $\zeta(s)$ defined on $(0,\frac{9}{8}\sqrt{n-2}\ ]$, satisfying
\begin{eqnarray*}
&&\frac{d}{ds}\left[\big((n-2)s^{-2}-1\big)^{-1}\zeta(s)\right]
\\
&=&\big((n-2)s^{-2}-1\big)^{-2}\left(2(n-2)^3 s^{-3}-5(n-2)^{\frac 7 2} s^{-6}- \frac{1}{2} (n-2)^{-13} s^{27}\right).
\end{eqnarray*}

Since we have
\begin{eqnarray*}
&&\big((n-2)s^{-2}-1\big)^{-2}\left(2(n-2)^3 s^{-3}-5(n-2)^{\frac 7 2} s^{-6}- \frac{1}{2} (n-2)^{-13} s^{27}\right)
\\
&=&-5(n-2)^{\frac 3 2}s^{-2}+O(1),
\end{eqnarray*}
when $s\rightarrow 0$, and
\begin{eqnarray*}
&&\big((n-2)s^{-2}-1\big)^{-2}\left(2(n-2)^3 s^{-3}-5(n-2)^{\frac 7 2} s^{-6}- \frac{1}{2} (n-2)^{-13} s^{27}\right)
\\
&=&\frac{1}{2}(n-2)^{\frac{3}{2}} \left(n-\frac{19}{4} \right)(\sqrt{n-2}-s)^{-2}+O(1),
\end{eqnarray*}
as $s\to \sqrt{n-2}$. 
It follows that
\begin{eqnarray*}
\zeta(s)=5(n-2)^{\frac{5}{2}}s^{-3}+O(s^{-2}), \text{ as } s\rightarrow 0,
\end{eqnarray*}
and that $\zeta(s)$ is smooth at $s=\sqrt{n-2}$, and
\begin{eqnarray*}
\zeta(\sqrt{n-2})=(n-2)\left(n-\frac{19}{4}\right)
\end{eqnarray*}

\begin{Lemma}
There exists a large number $N<\infty$, depending only on $n$, such that the following holds. Let
\begin{eqnarray*}
\psi_a(s)=\varphi(as)-(n-2)a^{-2}+a^{-4}\zeta(s)
\end{eqnarray*}
for $s\in[Na^{-1},\frac{9}{8}\sqrt{n-2}\ ]$, then
\begin{eqnarray*}
&&\psi_a(s)\psi_a''(s)-\frac{1}{2}(\psi_a'(s))^2+(n-2)\frac{\psi_a'(s)}{s}-\frac{\psi_a(s)\psi_a'(s)}{s}\\
&&+\frac{2(n-2)}{s^2}\psi_a(s)(1-\psi_a(s))-s\psi_a'(s)<0
\end{eqnarray*}
for $s\in[Na^{-1},\frac{9}{8}\sqrt{n-2}\ ]$ and for all $a$ large enough.
\end{Lemma}
\begin{proof}
$\zeta$ satisfies the following equation
\begin{eqnarray*}
(n-2)s^{-2}\big(s\zeta'(s)+2\zeta(s)\big)-s\zeta'(s)&=&s\big((n-2)s^{-2}-1\big)^{2}\frac{d}{ds}\left[\big((n-2)s^{-2}-1\big)^{-1}\zeta(s)\right]
\\
&=& 2(n-2)^3 s^{-2} -5(n-2)^{\frac{7}{2}}s^{-5} -\frac{1}{2}(n-2)^{-13}s^{28}.
\end{eqnarray*}
With this we compute
\begin{eqnarray*}
&&(n-2)s^{-2}\big(s\psi_a'(s)+2\psi_a(s)\big)-s\psi_a'(s)
\\
&=&(n-2)s^{-2}\big(as\varphi'(as)+2\varphi(as)\big)-as\varphi'(as)
-2(n-2)^2a^{-2}s^{-2}
\\
&&+\left(2(n-2)^3 a^{-4}s^{-2}-5(n-2)^{\frac 7 2} a^{-4}s^{-5}-\frac{1}{2}(n-2)^{-13}  a^{-4} s^{28}\right)
\\
&=&(n-2)s^{-2}(as\varphi'(as)+2\varphi(as))-4(n-5)(n-2)^3a^{-4}s^{-4}
\\ 
&&+\left(2(n-2)^3 a^{-4}s^{-2}-5(n-2)^{\frac 7 2} a^{-4}s^{-5}-\frac{1}{2}(n-2)^{-13}  a^{-4} s^{28}\right)+O(a^{-6}s^{-6}),
\end{eqnarray*}

\begin{eqnarray*}
&&\psi_a(s)\psi_a''(s)-\frac{1}{2}(\psi_a'(s))^2-s^{-2}\psi_a(s)\big(s\psi_a'(s)+2(n-2)\psi_a(s)\big)
\\
&=&a^2\varphi(as)\varphi''(as)-\frac{1}{2}a^2(\varphi'(as))^2-s^{-2}\varphi(as)\big(as\varphi'(as)+2(n-2)\varphi(as)\big)
\\
&&-(n-2)\varphi''(as)-2(n-2)^3a^{-4}s^{-2}+(n-2)a^{-2}s^{-2}\big(as\varphi'(as)\\
&& +4(n-2)\varphi(as)\big)+O(a^{-6}s^{-7})
\\
&=&a^2\varphi(as)\varphi''(as)-\frac{1}{2}a^2(\varphi'(as))^2-s^{-2}\varphi(as)\big(as\varphi'(as)+2(n-2)\varphi(as)\big)
\\
&&-2(n-2)^3a^{-4}s^{-2}+4(n-4)(n-2)^3a^{-4}s^{-4}+O(a^{-6}s^{-7}).
\end{eqnarray*}
Adding them up and using the equation of $\varphi(r)$, we have
\begin{eqnarray*}
&&\psi_a(s)\psi_a''(s)-\frac{1}{2}(\psi_a'(s))^2+(n-2)\frac{\psi_a'(s)}{s}-\frac{\psi_a(s)\psi_a'(s)}{s}\\
&&+\frac{2(n-2)}{s^2}\psi_a(s)(1-\psi_a(s))-s\psi_a'(s)
\\
&=& 4(n-2)^3a^{-4}s^{-4}-5(n-2)^{\frac 7 2} a^{-4}s^{-5}-\frac{1}{2}(n-2)^{-13} a^{-4}s^{28}+O(a^{-6}s^{-7})
\end{eqnarray*}
for all $s\in[r_*a^{-1},\frac{9}{8}\sqrt{n-2}\ ]$. Apparently, if we fix $N$ large enough, then the right-hand-side of the above equation is negative for all $s\in[Na^{-1},\frac{9}{8}\sqrt{n-2}\ ]$, this completes the proof.
\end{proof}

Next, we construct the barrier on $[r_*a^{-1}, N a^{-1}]$. Let $\beta_a(r)$ be the solution to the following equation
\begin{eqnarray*}
&&\beta_a(r)\varphi''(r)+\beta_a''(r)\varphi(r)-\varphi'(r)\beta_a'(r)+\frac{n-2}{r}\beta_a'(r)
\\
&&-\frac{1}{r}\big(\beta_a(r)\varphi'(r)+\beta_a'(r)\varphi(r)\big)+\frac{2(n-2)}{r^2}\big(1-2\varphi(r)\big)\beta_a(r)=-1
\end{eqnarray*}
with prescribed condition
\begin{eqnarray*}
\beta_a(N)&=&a^{-3}\zeta(a^{-1}N)-(n-2)a^{-1},
\\
\beta_a'(N)&=&a^{-4}\zeta'(a^{-1}N).
\end{eqnarray*}
Seeing that the above prescribed data for $\beta_a$ are bounded independent of $a$, and $\varphi$ is a smooth function on $[r_*,N]$, we have that $\beta_a$, $\beta_a'$, and $\beta_a''$ are bounded independent of $a$ on the interval $[r_*,N]$.
\begin{Lemma}
Let $\psi_a(s)=\varphi(as)+a^{-1}\beta_a(as)$ for all $s\in[r_*a^{-1},Na^{-1}]$. We have that for all $a$ large enough
\begin{eqnarray*}
&&\psi_a(s)\psi_a''(s)-\frac{1}{2}(\psi_a'(s))^2+(n-2)\frac{\psi_a'(s)}{s}-\frac{\psi_a(s)\psi_a'(s)}{s}\\
&&+\frac{2(n-2)}{s^2}\psi_a(s)(1-\psi_a(s))-s\psi_a'(s)<0
\end{eqnarray*}
whenever $s\in[r_*a^{-1},Na^{-1}]$.
\begin{proof}
By using the equation of $\varphi$ and $\beta_a$, we have
\begin{align*}
&\psi_a(s)\psi_a''(s)-\frac{1}{2}(\psi_a'(s))^2+(n-2)\frac{\psi_a'(s)}{s}-\frac{\psi_a(s)\psi_a'(s)}{s}+\frac{2(n-2)}{s^2}\psi_a(s)(1-\psi_a(s))
\\
=& a^2\left[\vp(as) \vp''(as)-\frac{1}{2}\left(\vp'(as)\right)^2 +(n-2)\frac{\vp'(as)}{as} -\frac{\vp(as)\vp'(as)}{(as)^2} \right. \\
& +\left. \frac{2(n-2)}{(as)^2}\vp(as) \left(1-\vp(as) \right) \right] \\
& +a\left[ \b_a(as)\vp''(as)+\b_a''(as)\vp(as)-\vp'(as)\b_a'(as)+\frac{n-2}{as}\b_a'(as)  \right. \\
&  \left. -\frac{1}{as}\left(\b_a(as)\vp'(as)+\b_a'(as)\vp(as)\right)+\frac{2(n-2)}{(as)^2}(1-2\vp(as))\b_a(as) \right] \\
& +\beta_a(as)\beta_a''(as)-\frac{1}{2}(\beta_a'(as))^2-\frac{1}{as}\beta_a(as)\beta_a'(as)-\frac{2(n-2)}{(as)^2}(\beta_a(as))^2
\\
=& -a + \beta_a(as)\beta_a''(as)-\frac{1}{2}(\beta_a'(as))^2-\frac{1}{as}\beta_a(as)\beta_a'(as)-\frac{2(n-2)}{(as)^2}(\beta_a(as))^2
\\
\leq & -a+C.
\end{align*}
On the other hand
\begin{eqnarray*}
s\psi_a'(s)=as\varphi'(as)+s\beta_a'(as)\geq-C,
\end{eqnarray*}
the conclusion follows immediately if we take $a\gg C$ large enough.
\end{proof}
\end{Lemma}

We summarize the above two lemmas in the following proposition.

\begin{Proposition}
The function
\begin{eqnarray*}
\psi_a(s)=\left\{\begin{array}{ll}
\varphi(as)-(n-2)a^{-2}+a^{-4}\zeta(s),&\text{ for } s\in[Na^{-1},\frac{9}{8}\sqrt{n-2}\ ]
\\
\varphi(as)+a^{-1}\beta_a(as), &\text{ for } s\in[r_*a^{-1},Na^{-1}]\end{array}\right.
\end{eqnarray*}
is continuously differentiable, and satisfies the inequality
\begin{eqnarray*}
&&\psi_a(s)\psi_a''(s)-\frac{1}{2}(\psi_a'(s))^2+(n-2)\frac{\psi_a'(s)}{s}-\frac{\psi_a(s)\psi_a'(s)}{s}\\
&&+\frac{2(n-2)}{s^2}\psi_a(s)(1-\psi_a(s))-s\psi_a'(s)<0,
\end{eqnarray*}
for all $s\in[r_*a^{-1},\frac{9}{8}\sqrt{n-2}\ ]$ and for all $a$ large enough.
\end{Proposition}

\begin{Proposition}
There is a small positive constant $\theta$, depending only on $n$, such that
\begin{eqnarray*}
\psi_a(s)\geq (n-2)a^{-2}\big((n-2)s^{-2}-1\big)+\frac{n-2}{16}a^{-4}
\end{eqnarray*}
for all $s\in[\sqrt{n-2}-\theta,\sqrt{n-2}+\theta]$. In particular, if $a$ is sufficiently large, then $\psi_a(s)\geq\frac{n-2}{32}a^{-4}$ for all $s\in[r_*a^{-1},\sqrt{n-2}(1+\frac{1}{100}a^{-2})]$.
\end{Proposition}

\begin{proof}
Since $\zeta(\sqrt{n-2})-(n-5)(n-2)=(n-2)(n-\frac{19}{4})-(n-5)(n-2)=\frac{n-2}{4}$, we can chose $\theta>0$ small enough such that $\zeta(s)-(n-5)(n-2)^3 s^{-4}\geq\frac{n-2}{8}$ for all $s\in[\sqrt{n-2}-\theta,\sqrt{n-2}+\theta]$. 
Therefore
\begin{eqnarray*}
\psi_a(s)&=&\varphi(as)-(n-2)a^{-2}+a^{-4}\zeta(s)
\\
&=&(n-2)a^{-2}\left((n-2)s^{-2}-1\right)+a^{-4}\left(\zeta(s)-(n-5)(n-2)^3s^{-4}\right)+O(a^{-6})
\\
&\geq& (n-2)a^{-2}\big((n-2)s^{-2}-1\big)+\frac{n-2}{8}a^{-4}+O(a^{-6}),
\end{eqnarray*}
for all $s\in[\sqrt{n-2}-\theta,\sqrt{n-2}+\theta]$. Hence, if $a$ is sufficiently large, we have
\begin{eqnarray*}
\psi_a(s)\geq(n-2)a^{-2}\big((n-2)s^{-2}-1\big)+\frac{n-2}{16}a^{-4}
\end{eqnarray*}
for all $s\in[\sqrt{n-2}-\theta,\sqrt{n-2}+\theta]$. In particular $\psi_a(s)\geq\frac{n-2}{32}a^{-4}$ for all $s\in[\sqrt{n-2}-\theta,\sqrt{n-2}(1+\frac{1}{100}a^{-2})]$, when $a$ is large enough. On the other hand, it is easy to see that $\psi_a(s)\geq\frac{n-2}{32}a^{-4}$ for all $s\in[r_*a^{-1},\sqrt{n-2}-\theta]$ if $a$ is sufficiently large. This completes the proof.
\end{proof}

Since the ancient solution is of Type II, we have that the size of non-neck-like region must be small compared to $\sqrt{-t}$. As in \cite{brendle2018ancient}, we also have the following.

\begin{Lemma}[Lemma 2.7 in \cite{brendle2018ancient}]
Given any $\delta>0$, we have $\liminf_{t\rightarrow-\infty}\sup_{r\geq\delta\sqrt{-t}}u(r,t)=0$.
\end{Lemma}

\begin{Proposition}[Proposition 2.8 in \cite{brendle2018ancient}]
There exists a large constant $K$ such that the following holds. Suppose $a\geq K$ and $\bar{t}\leq 0$. Suppose that $\bar{r}(t)$ is a function satisfying $ \Big|\frac{\bar{r}}{\sqrt{-2(n-2)t}}-1\Big|\leq\frac{1}{100}a^{-2}$ and $u(\bar{r}(t),t)\leq\frac{n-2}{32}a^{-4}$ for all $t\leq\bar{t}$. Then $u(r,t)\leq \psi_a(\frac{r}{\sqrt{-2t}})$ for all $t\leq\bar{t}$ and $r_*a^{-1}\sqrt{-2t}\leq r\leq \bar{r}(t)$. In particular $u(r,t)\leq Ca^{-2}$ for all $t\leq\bar{t}$ and $\frac{1}{2}\sqrt{-2(n-2)t}\leq r\leq\bar{r}(t)$.
\end{Proposition}

\begin{Proposition}[Proposition 2.9 in \cite{brendle2018ancient}]\label{prop:barrier_improvement}
Suppose there exists $\bar{r}(t)$ such that $\bar{r}(t)=\sqrt{-2(n-2)t}+O(1)$ and $u(\bar{r}(t),t)\leq O(\frac{1}{-t})$ as $t\rightarrow \infty$. Then we can find a large constant $K\geq 100n^2$, such that $u(r,t)\leq\psi_a(\frac{r}{\sqrt{-2t+Ka^2}})$ whenever $a\geq K$, $t\leq -K^2a^2$, and $r_*a^{-1}\sqrt{-2t+Ka^2}\leq r\leq\bar{r}(t)$. In particular, if $t=-K^2a^2$, then the set $\{r\leq\frac{1}{2}\sqrt{-2(n-2)t}\}\subset (M,g(t))$ has diameter at least $\sim (-t)$.
\end{Proposition}

\subsection{Uniqueness of rotationally symmetric $\kappa$-solutions}

Once we have the valid barriers constructed in the last subsection, the remaining arguments are the same as section 3 and section 4 in Brendle \cite{brendle2018ancient}. In this section we will outline the main ideas of the proof of Theorem \ref{Thm:symmetric}, and it is left to the readers to check the details following \cite{brendle2018ancient}.

We fix a point $q$ in the rotationally symmetric $\kappa$-solution such that $q$ is the center of an $\varepsilon$-neck when $t=0$; the existence of such $q$ is guaranteed by Proposition \ref{prop:asymptotic_cylinder}. Let $\bar{r}(t)$ denote the radius of the product factor $\mathbb{S}^{n-1}$ that passes through $q$ at time $t$. By using the neck stability of Kleiner-Lott \cite{kleiner2014singular} (Theorem \ref{Thm:neck_stability} in our case), one can show that $(q,t)$ is always the center of an $\varepsilon$-neck and that the scaled limit along $(q,t)$ is the standard cylinder as the asymptotic shrinker, and this implies when $t\rightarrow-\infty$
\begin{eqnarray}\label{eq:r_bar}
\frac{\bar{r}(t)}{\sqrt{-2(n-2)t}}\rightarrow 1.
\end{eqnarray}
Furthermore, since $q$ is a fixed point in the Ricci flow, we have
\begin{eqnarray*}
\frac{d}{dt}\bar{r}(t)=-v(\bar{r},t)=-\frac{1}{\bar{r}}\left((n-2)\big(1-u(\bar{r},t)\big)-\frac{1}{2}\bar{r}u_r(\bar{r},t)\right).
\end{eqnarray*}

Let the function $F(z,t)$ be defined by
\begin{eqnarray*}
F\left(\int_{\bar{r}(t)}^\rho u^{-\frac{1}{2}}(r,t)dr,t\right)=\rho.
\end{eqnarray*}
In other words, $F(z,t)$ stands for the radius of the sphere that has distance $z$ from the point $q$. By anaylizing the equations satisfied by $F(z,t)$ around $(0,t)$ when $-t$ is large, one obtains a more accurate behavior of $\bar{r}(t)$ than (\ref{eq:r_bar}). Now we summarize some properties of $F(z,t)$.

\begin{Proposition}[Proposition 3.3 in \cite{brendle2018ancient}]
The function $F$ satisfies the equation
\begin{eqnarray*}
0&=&F_t(z,t)-F_{zz}(z,t)+\frac{n-2+(F_z(z,t))^2}{F(z,t)}
\\
&&+(n-1)F_z(z,t)\left(-F_z(0,t)F(0,t)^{-1}+\int_{F(0,t)}^{F(z,t)}\frac{u^{\frac{1}{2}}(r,t)}{r^2}dr\right).
\end{eqnarray*}
\end{Proposition}

\begin{Corollary}[Corollary 3.4 in \cite{brendle2018ancient}]
\begin{eqnarray*}
&&\left|F_t(z,t)-F_{zz}(z,t)+\frac{n-2+(F_z(z,t))^2}{F(z,t)}\right|
\\
&\leq&(n-1)F(0,t)^{-1}F_z(0,t)F_z(z,t)
\\
&&+ (n-1)\max\left\{F_z(z,t),F_z(0,t)\right\}\left|\frac{1}{F(z,t)}-\frac{1}{F(0,t)}\right|F_z(z,t).
\end{eqnarray*}
\end{Corollary}

Let $\tau<0$, we define the following rescaling of $F$.
\begin{eqnarray*}
G(\xi,\tau):=e^{\frac{\tau}{2}}F(e^{-\frac{\tau}{2}}\xi,-e^{-\tau})-\sqrt{2(n-2)},
\end{eqnarray*}
we then have $G_\xi(\xi,\tau)>0$ and $G_{\xi\xi}(\xi,\tau)\leq 0$, since $u^{\frac{1}{2}}>0$ and $u_r\leq0$, respectively. Furthermore, $G(\xi,\tau)\rightarrow 0$ locally smoothly as $\tau\rightarrow-\infty$ because of (\ref{eq:r_bar}).

\begin{Proposition}[Proposition 3.6 in \cite{brendle2018ancient}]
$G(\xi,\tau)$ satisfies 
\begin{eqnarray*}
&&\Bigg|G_\tau(\xi,\tau)-G_{\xi\xi}(\xi,\tau)+\frac{1}{2}\xi G_{\xi}(\xi,\tau)-G(\xi,\tau)+\frac{(G_\xi(\xi,\tau))^2+\frac{1}{2}(G(\xi,\tau))^2}{\sqrt{2(n-2)}+G(\xi,\tau)}\Bigg|
\\
&=&\Bigg|G_\tau(\xi,\tau)-G_{\xi\xi}(\xi,\tau)+\frac{1}{2}\xi G_{\xi}(\xi,\tau)-\frac{1}{2}\big(\sqrt{2(n-2)}+G(\xi,\tau)\big)+\frac{n-2+(G_\xi(\xi,\tau))^2}{\sqrt{2(n-2)}+G(\xi,\tau)}\Bigg|
\\
&\leq&\frac{n-1}{\sqrt{2(n-2)}+G(0,\tau)}G_\xi(0,\tau)G_\xi(\xi,\tau)
\\
&&+(n-1)\Bigg|\frac{1}{\sqrt{2(n-2)}+G(\xi,\tau)}-\frac{1}{\sqrt{2(n-2)}+G(0,\tau)}\Bigg|\max\left\{G_\xi(\xi,\tau),G_\xi(0,\tau)\right\}G_\xi(\xi,\tau).
\end{eqnarray*}
\end{Proposition}
It is important to see that the principal part in the equation of $G$ is also
\begin{eqnarray*}
G_\tau(\xi,\tau)-G_{\xi\xi}(\xi,\tau)+\frac{1}{2}\xi G_{\xi}(\xi,\tau)-G(\xi,\tau).
\end{eqnarray*}
Therefore, one may follow the same argument as in \cite{brendle2018ancient} to obtain 
\begin{eqnarray*}
|G(0,\tau)|\leq O(e^{\frac{\tau}{2}}), & G_\xi(0,\tau)\leq O(e^{\frac{\tau}{2}}),
\end{eqnarray*}
and consequently
\begin{eqnarray*}
\left|\bar{r}(t)-\sqrt{-2(n-2)t}\right|\leq O(1), & u(\bar{r}(t),t)\leq O(\frac{1}{-t}).
\end{eqnarray*}
Applying Proposition \ref{prop:barrier_improvement} to $\bar{r}(t)$ one obtains the following.

\begin{Proposition}[Proposition 3.11 in \cite{brendle2018ancient}]
For $-t$ large, the set $\{r\leq\frac{1}{2}\sqrt{-2(n-2)t}\}\subset (M,g(t))$ has diameter $\sim (-t)$. Moreover, $\lim_{t\rightarrow-\infty}R_{\max}(t)>0$
\end{Proposition}

After showing that
\begin{eqnarray*}
\mathcal{R}:=\lim_{t\rightarrow-\infty}R_{\max}(t)>0,
\end{eqnarray*}
we proceed to work with the quantity $R+u^{-1}v^2$. Note that if we are on a steady soliton, that is, if $u$ is constant in time, then $R+u^{-1}v^2=R+|V|^2\equiv const$ everywhere. We collect some equations satisfied by $R+u^{-1}v^2$, they are analogous to the correspondent equations in \cite{brendle2018ancient} and their proofs are straightforward computations.

\begin{Proposition}
We have the following equations and inequalities.
\begin{eqnarray}
&&R_t-\frac{n-1}{r}u^{-1}u_tv\geq 0.
\\
&&(R+u^{-1}v^2)_t+\frac{v}{n-1}\left(1+\frac{r}{n-1}u^{-1}v\right)^{-1}(R+u^{-1}v^2)_r\geq 0
\\
&&(R+u^{-1}v^2)_r=-\frac{n-1}{r}\left(1+\frac{r}{n-1}u^{-1}v\right)u^{-1}u_t
\\
&&(R+u^{-1}v^2)_t=u(R+u^{-1}v^2)_{rr}+\frac{n-1}{r}u(R+u^{-1}v^2)_r+O(r)(R+u^{-1}v^2)_r.
\end{eqnarray}
\end{Proposition}

Following the same argument as in \cite{brendle2018ancient} one may obtain
\begin{eqnarray*}
R+u^{-1}v^2&\geq&\mathcal{R},
\\
(R+u^{-1}v^2)_r&\geq&0,
\end{eqnarray*}
everywhere in space time and furthermore
\begin{eqnarray*}
\lim_{r\rightarrow\infty}r^2 u(r,t)=(n-2)^2\mathcal{R}^{-1},
\end{eqnarray*}
whenever $-t$ is large enough. Since we also have $\lim_{r\rightarrow\infty} u(r,t)=0$ and $\lim_{r \rightarrow\infty} ru_r(r,t)=0$, where the second equation follows from the fact that the radial component of the Ricci curvature is very small compared to its orbital components on an $\varepsilon$-neck, by using the identity
\begin{eqnarray*}
R+u^{-1}v^2=\frac{u^{-1}}{r^2}\left(n-2+u-\frac{1}{2}ru_r\right)^2-\frac{n-1}{r^2}(n-2+u),
\end{eqnarray*}
we obtain that when $-t$ is large
\begin{eqnarray*}
\lim_{r\rightarrow\infty}R+u^{-1}v^2=\frac{(n-2)^2}{(n-2)^2\mathcal{R}^{-1}}=\mathcal{R}.
\end{eqnarray*}
Summarizing the above facts, we have
\begin{eqnarray*}
R+u^{-1}v^2\equiv \mathcal{R},
\end{eqnarray*}
when $-t$ is large. Finally, since
\begin{eqnarray*}
(R+u^{-1}v^2)_r=-\frac{n-1}{r}\left(1+\frac{r}{n-1}u^{-1}v\right)u^{-1}u_t,
\end{eqnarray*}
we have $u_t\equiv 0$ when $-t$ is large. This implies the $\kappa$-solution is a rotationally symmetric steady soliton, hence the Bryant soliton.

\bigskip

\emph{Acknowledgement.} 
We would like to thank Professors Bennett Chow, Lei Ni, and Richard Schoen for helpful discussions, encouragements, and supports. We also thank the referee for corrections and suggestions which improve the exposition of the paper.

\bibliographystyle{plain}
\bibliography{citation}




\end{document}